\documentclass[11pt]{amsart}
\usepackage[T1]{fontenc}

\usepackage{lmodern, amsfonts,amsmath,amstext,amsbsy,amssymb,
amsopn,amsthm,upref,eucal,mathptmx}

\usepackage{mathrsfs}

\RequirePackage{xcolor} 
\definecolor{halfgray}
{gray}{0.55}
\definecolor{webgreen}
{rgb}{0,0.4,0}
\definecolor{webbrown}
{rgb}{.8,0.1,0.1}
\definecolor{red}
{rgb}{1,0,0}
\usepackage{microtype}

\newcommand \R {{ \mathbb R}}
\def\C{{\mathbb C}}
\newcommand \Z {{ \mathbb Z}}
\newcommand \N {{ \mathbb N}}
\newcommand \T {{ \mathbb T}}

\newcommand \re {{%
\operatorname{Re}
}}
\newcommand \im {{%
\operatorname{Im}
}}
\newcommand{\SL}{%
\operatorname{SL}
}

\newcommand{\<}{{\langle}} 
\renewcommand{\>}{{\rangle}}

\newcommand{\Cal}{\mathcal}

\newtheorem{theorem}{Theorem}[section]
\newtheorem {lemma} [theorem]{Lemma}
\newtheorem {proposition}[theorem]{Proposition}
\newtheorem{corollary}[theorem]{Corollary}
\newtheorem{remark}[theorem]{Remark}

\newtheorem{definition}[theorem]{Definition}

\title[Twisted Cohomological Equations for Translation Flows]%
{Twisted Cohomological Equations \\ for  Translation Flows}

  \author{Giovanni Forni}

\address{Department  of Mathematics\\
  University of Maryland \\
  College Park, MD USA}
  
\email
    {gforni@math.umd.edu}
\keywords
      {Translation surfaces and flows, twisted cohomological equations, twisted invariant distributions, 
      Sobolev estimates.}
\subjclass[2010]
        {37C10, 37E35, 37A20, 30E20, 31A20.}

\date{\today}

\begin{document}

\def\echo#1{\relax}
    
\begin{abstract}
  \begin{sloppypar}
  We prove by methods of harmonic analysis a result on existence of solutions for twisted cohomological equations on translation surfaces with loss of derivatives at most $3+$ in Sobolev spaces. As a consequence we prove that product translation flows on ($3$-dimensional) translation manifolds which are products of a  (higher genus) translation surface with a (flat) circle are stable in the sense of A.~Katok. In turn, our result on product flows implies a stability result of time-$\tau$ maps of translation flows on translation surfaces. 
   \end{sloppypar}
\end{abstract}

 \maketitle
\date{\today}

 \section{Introduction}
\noindent The first result on solutions of the cohomological equation for a parabolic non-homogeneous
 (but ``locally homogeneous'') smooth flow was given by the author in~\cite{F97}  in the case of translation flows
 (and their smooth time changes) on higher genus translation surfaces, by method of  harmonic analysis based on the theory of boundary behavior of holomorphic functions. 
 
\noindent  Since then refined versions of that result have been proved by  (dynamical)
 renormalization methods based on ``spectral gap''  (and hyperbolicity) properties of the Rauzy-Veech-Zorich cocycle \cite{MMY05}, \cite{MY16}, of the Kontsevich--Zorich cocycle over the Teichm\"uller flow \cite{F07} or, more recently,   of the transfer operator of a pseudo-Anosov map on appropriate anisotropic Banach space of currents \cite{FGL}.  The renormalization approach has the immediate advantage of a refined control on the regularity loss and of more explicit  conditions of Diophantine type on the dynamics, and in particular applies to self-similar translation flows or Interval Exchange Transformations. It also gives a direct approach to results for almost all translation surfaces, while an extension to almost all directions {\it for any given
 translation surface}  had to wait for the work of J.~Chaika and A.~Eskin \cite{CE} based on the breakthrough of 
 A.~Eskin, M.~Mirzakhani  \cite{EM}, A.~Eskin, M.~Mirzakhani and A.~Mohammadi \cite{EMM} and 
 S. Filip~\cite{Fi}.  
 
 \medskip
\noindent   In this paper we apply a twisted version of the arguments of~\cite{F97} to the solution of the  twisted cohomological equation for translation flows and derive results on the cohomological
 equation for $3$-dimensional  ``translation flows'' on products of a higher genus translation surface 
 with a circle.  For these problems no renormalization approach is available at the moment, although 
 steps in that direction have been made in the work of A.~Bufetov and B.~Solomyak \cite{BS14}, \cite{BS18a},
 \cite{BS18b}, \cite{BS18c}, \cite{BS19} and of the author \cite{F19}, who  have introduced twisted versions of the Rauzy--Veech--Zorich and Kontsevich--Zorich cocycles, respectively,  and proved ``spectral gap'' results for them.
 
 \medskip
 \noindent For any translation surface $(M,h)$ (a pair of a Riemann surface $M$ and an Abelian differential
 $h$ on $M$ let $H^s_h(M)$ denote the scale of (weighted) Sobolev spaces (introduced in \cite{F97},
 and recalled below in \S \ref{sec:analysis}). For the horizontal  translation flow $\phi_\R^h$ on $M$ (of 
 generator the horizontal vector field $S$) and for any $\sigma \in \R$, let $\Cal I^s_{h, \sigma} \subset 
 H^{-s}_h(M)$  denote the space of $(S+ \imath \sigma)${\it -invariant distributions}, that is, the subspace
 $$
 \Cal I^s_{h, \sigma}:=\{ D\in H^{-s}_h(M) \vert  (S+ \imath \sigma) D=0 \in H^{-(s+1)}_h(M)\}\,.
 $$
For all $\theta\in \T$, let $h_\theta:= e^{-\imath \theta} h$ denote the rotated Abelian differential and let 
$S_\theta$ denote the generator of the horizontal translation flow on $(M, h_\theta)$.

\smallskip
 \noindent We prove the following results.
 
 \begin{theorem}  
 \label{thm:main} For any translation surface $(M,h)$, for almost all $\theta \in \T$ and for almost all $\sigma \in \R$ (with respect to the Lebesgue measure)  the following holds. 
 For all $f \in H^s_h(M)$  with $s>3$, satisfying the distributional conditions $D(f)=0$ for all $(S_\theta + \imath \sigma)$-invariant distributions $D \in H^{-s}_h(M)$, the twisted cohomological equation $(S_\theta+ \imath \sigma) u =f$ has a solution $u\in H^r_h(M)$ for all $r< s-3$, and there exists a constant  $C_{r,s} (\theta, \sigma) >0$ such that
 $$
 \vert u  \vert_r  \leq   C_{r,s} (\theta, \sigma) \vert  f \vert_s\,.
 $$
  \end{theorem} 
\noindent  In other words, the theory of the twisted cohomological equation of translation flows is analogous,
 {\it for Lebesgue almost all twisting parameters}, to the untwisted theory of the cohomological equation
 for translation flows.
 
 \begin{remark} In the untwisted case the optimal loss of regularity of solutions of the cohomological equation  is known to be $1+$ for $L^2$ Sobolev norms, for almost all translation flows with respect to any 
 $SL(2,\R)$ invariant measure under the hypothesis of hyperbolicity of the KZ cocycle~\cite{F07}.  Marmi and Yoccoz~\cite{MY16} proved a similar, but slightly weaker,
 statement for H\"older norms. For self-similar translation flows, the loss of $1+$ derivatives for  H\"older norms should follow from the recent work of Faure, Gou\"ezel and Lanneau~\cite{FGL}, although spaces with fractional exponents are not explicitly considered in their paper. 
 
\noindent It is natural to conjecture that the optimal loss of derivatives is $1+$ also in the twisted case, and it seems plausible that the whole argument of \cite{F07} would carry over under the (equivalent ?)
 hypotheses of hyperbolicity of the  twisted cocycles introduced in~\cite{BS18c} and \cite{F19}.  At the moment the only known results on such twisted exponents are upper bounds (in particular that the top exponent is $<1$)  \cite{BS18c}, \cite{BS19}, \cite{F19}  but no lower bounds are known.
 \end{remark} 
 
\noindent  A result on the existence of solutions of the cohomological equation for twisted horocycle flows was recently proved by L.~Flaminio, the author and J.~Tanis \cite{FFT16}, who were motivated by applications to the cohomological equation for horocycle time-$\tau$ maps (see also \cite{Ta12}) and to deviation of ergodic averages for twisted horocycle integrals and  horocycle time-$\tau$ maps.  Twisted nilflows are still nilflows so the theory of twisted cohomological equations  in the nilpotent case is covered by the general results of L.~Flaminio and the author~\cite{FlaFo07}. As for results on deviation of ergodic averages for nilflows, they are related to bounds on Weyl sums for polynomials. The Heisenberg (and the general step 2) case are better understood by renormalization methods (see for instance \cite{FlaFo06}), while the higher step case is not renormalizable, hence harder (see for instance \cite{GT12}, \cite{FlaFo14}).  Results on twisted ergodic integrals of translation flows and applications to effective weak mixing were recently proved by the author \cite{F19}. 
 
 \bigskip
\noindent  For all $(s, \nu) \in \R^+\times \N$,  let $H^{s, \nu} (M\times \T)$ denote the $L^2$ Sobolev space  on $M\times \T$ with respect to  the invariant volume form $\omega_h \wedge d\phi$ and the vector fields $S$, $T$, and $\partial/\partial \phi$:  for all $s, \nu \geq 0$, we define
 $$
 H^{s,\nu} (M\times\T):= \{ f \in L^2(M\times \T, d \text{vol}) \vert   \sum_{i+j\leq s}\sum_{\ell \leq \nu} \Vert S^i T^j \frac{\partial^\ell f}{\partial \phi^\ell} \Vert_0 < +\infty\}\,;
 $$
the space  $H^{-s, -\nu}_h(M\times\T)$ is defined as the dual space of $H^{s,\nu}(M\times\T)$.

\noindent The space $L^2(M\times \T, d \text{vol})$ of the product manifold with respect to
the invariant volume form $\omega_h \wedge d\phi$ decomposes as a direct sum of the eigenspaces
$\{H^0_n \vert n\in \Z\}$ of the circle action:
$$
L^2(M\times \T, d \text{vol}) = \bigoplus_{n\in \Z}  H^0_n \,.
$$
Let now $X_{\theta, c}= S_\theta +  c  \frac{\partial}{\partial \phi}$ denote a translation vector field
on $M \times \T$, and let $\Cal I^{s,\nu}_{h_\theta, c} \subset H^{-s, -\nu}_h(M\times\T)$ denote the space of $X_{\theta, c}$-invariant distributions.  The subspace of $X_{\theta, c}$-invariant distributions
in $\Cal I^{s,\nu}_{h_\theta, c}$ supported on the Sobolev subspace of $H_n^s \subset H_n^0$ has finite
and non-zero dimension, uniformly bounded with respect to $n\in \N$. It follows that the space  $\Cal I^{s,\nu}_{h_\theta, c}$ has countable dimension. 

 \begin{theorem}  
 \label{thm:main_bis} 
 For any translation surface $(M,h)$, for almost all $\theta \in \T$ and for almost all $c \in \R$ (with respect to the Lebesgue measure)  the following holds. For all $f \in H^{s,\nu}_h(M \times \T)$  with $s>3$ and $\nu>2$,  satisfying the distributional conditions $D(f)=0$ for all $X_{\theta,c}$-invariant distributions $D \in \Cal I^{s,\nu}_{h_\theta,c} \subset H^{-s,-\nu}_h(M\times \T)$, the cohomological equation $X_{\theta,c} u =f$ has a solution $u\in H^{r, \mu}_h(M \times \T)$ for all $r< s-3$ and $\mu<\nu-2$, and there exists a constant  $C^{(\mu, \nu)}_{r,s} (\theta, c) >0$ 
 such that
 $$
 \Vert u  \Vert_{r, \mu}  \leq   C^{(\mu, \nu)}_{r,s} (\theta, c) \Vert  f \Vert_{s, \nu}\,.
 $$
  \end{theorem} 
  
 \noindent  Theorem \ref{thm:main_bis} states that   for almost all $(\theta, c)\in \T \times \R$, the flow of the vector field $X_{\theta,c}$ on $M\times \T$ is stable in the sense of A.~Katok.  In fact, ours  is the first example of a stable non-homogeneous (although locally homogeneous), non (partially) hyperbolic flow on a manifold of dimension at least $3$.
 Indeed, we recall that the only known examples of stable (and renormalizable)  $3$-dimensional flows are (up to smooth conjugacies and time-changes) homogeneous flows: horocycle flows of hyperbolic surfaces~\cite{FlaFo03} and  Heisenberg nilflows \cite{FlaFo06}.   However, there has been recent
progress (although conditional)  on ``Ruelle asymptotics''  and deviation of ergodic averages for horocycle flows for negatively curved metrics on surfaces \cite{Ad} (see also \cite{FG18}), hence a proof of smooth stability (at least in low regularity) for such flows seems within reach of current methods for the analysis for the transfer operator of hyperbolic flows, along the lines of the work of P.~Giulietti and C.~Liverani \cite{GL} for Anosov maps of tori.

\medskip
\noindent By a well-known argument we can derive from our result on the cohomological equation for the product flows a result on the cohomological equation for the time-$\tau$ maps of translation flows.  Let $\Phi^\tau_\theta$ denote the time-$\tau$ map of the horizontal translation flow of the Abelian differential $h_\theta$ on $M$. For all $s\geq 0$, let  $\Cal I^{s}_{h_\theta, \tau} \subset H^{-s}_h(M)$ denote the space of $\Phi^\tau_\theta$-invariant distributions, that is, the space of all distributions in $H^{-s}_h(M)$ which vanish  on the subspace
$$
\overline{ \{  u \circ \Phi^\tau_\theta - u \vert   u\in H^\infty_h(M)\} \cap H^s_h(M)}   \subset H^s_h(M) \,.
$$
We have the following result:

\begin{corollary}
\label{cor:main}
For any translation surface $(M,h)$, for almost all $\theta \in \T$ and for almost all $T \in \R$ (with respect to the Lebesgue measure)  the following holds.   For all $f \in H^{s}_h(M)$  with $s>3$,  satisfying the distributional conditions $D(f)=0$ for all $\Phi^\tau_\theta$-invariant distributions $D \in \Cal I^{s}_{h_\theta, \tau} \subset H^{-s}_h(M)$, the cohomological equation $ u\circ \Phi^\tau_\theta -u =f$ has a solution $u\in H^{r}_h(M)$, for all
$r< s-3$, and there exists a constant  $C_{r,s} (\theta, \tau) >0$  such that
 $$
 \vert u  \vert_r  \leq   C_{r,s} (\theta, \tau) \vert  f \vert_s\,.
 $$
\end{corollary}

\bigskip
\noindent The paper is organized as follows. In section \ref{sec:analysis} we recall basic facts of analysis on translation surfaces as developed by the author in \cite{F97} and \cite{F07}.  In section~\ref{sec:Beurling} we introduce a twisted version of the Beurling-type isometry of the $L^2$ space of a translation surface defined in \cite{F97}
(see also \cite{F02}).  Section~\ref{sec:spectral_unitary} recalls results from the theory of boundary behavior of Cauchy integrals of finite measures on the circle and applications to the spectral theory of general unitary operators on Hilbert spaces, following~\cite{F97}.  In section~\ref{sec:STCE} we prove the core result about the existence of  solutions of the cohomological equation, by the following the presentation given
in \cite{F07} of the original argument of \cite{F97}, generalized to the twisted case. Subsection \ref{DS}  is devoted to the core result about existence of distributional solutions, subsection~\ref{TIDBC} 
to finiteness result for the spaces of twisted invariant distributions and, finally, subsection \ref{SS} to the proof
of the main results on existence of smooth solutions for the twisted cohomological equations, the product flows
and the time-$\tau$ maps.
 
 \section*{Acknowledgements} \noindent  We wish to thank Pascal Hubert and Carlos Matheus for their interest in this
 work and their encouragement.
 This research was supported by the NSF grant DMS 1600687 and by a Research Chair of  the Fondation Sciences Math\'ematiques de Paris (FSMP). The author is grateful to the Institut de Math\'ematiques de Jussieu (IMJ) for its hospitality.
 
 \section{Analysis on translation surfaces}
 \label{sec:analysis}
 
 \noindent  This section gathers basic results on the flat Laplacian of  a translation surface, following~\cite{F97}, \S 2  and \S 3, and~\cite{F07}, \S 2.
 
Let $\Sigma_{h}:=\{p_{1},\dots,p_{\sigma}\}\subset M_h$ be the set of zeros of the holomorphic Abelian differential~$h$ on a Riemann surface $M$, of even orders $(k_{1},\dots,k_{\sigma})$ respectively with~$k_{1} + \dots + k_{\sigma}=2g-2$.  Let  $R_{h}:=\vert h\vert$ be the flat metric with cone singularities at $\Sigma_{h}$ induced by the Abelian  differential $h$ on $M$ and let $\omega_h$ denote its area form.  With respect to a holomorphic local coordinate $z=x+\imath y$ at a regular point, the Abelian differential 
$h$ has the form $h=\phi(z)dz$, where $\phi$ is a locally defined holomorphic function, and, consequently, 
\begin{equation}
\label{eq:metric}
 R_h= |\phi(z)| (dx^2 +dy^2)^{1/2}\,,\quad  \omega_h=|\phi(z)|^2\,dx\wedge dy\,.
\end{equation}
\noindent The metric $R_{h}$ is flat, degenerate at the finite set $\Sigma_{h}$ of zeroes of $h$ 
and has trivial holonomy, hence $h$ induces a structure of {\it translation 
surface}  on $M$. 

 \smallskip
  \noindent The weighted $L^{2}$ space is the standard space $L^{2}_{h}(M):= L^{2}(M,\omega_{h})$  with respect to the area element $\omega_{h}$  of the metric $R_{h}$. Hence the weighted $L^{2}$ 
  norm $\vert \cdot\vert_{0}$ are induced by the hermitian product $\<\cdot, \cdot\>_{h}$ defined as follows: for all functions $u$,$v\in L^{2}_{h}(M)$,
 \begin{equation}
 \label{eq:0norm}
 \< u ,v\>_{h} :=  \int _{M} u\,\bar v \, \omega_{h}\,\,.
 \end{equation}
 Let $\mathcal  F_{h}$ be the {\it horizontal foliation},  $\mathcal  F_{-h}$ be the {\it vertical foliation} for
the holomorphic Abelian differential $h$ on $M$. The foliations $\mathcal  F_{h}$ and $\mathcal  F_{-h}$
are measured foliations (in the Thurston's sense):  $\mathcal  F_{h}$ is the foliation given by the equation $\im h=0$ endowed with the invariant transverse measure 
$\vert \im h \vert$,  $\mathcal  F_{-h}$ is the foliation given  by the equation $\re h=0$ endowed with the 
invariant transverse measure $\vert \re h \vert$.  Since the metric $R_h$ is flat with trivial holonomy,  there exist  commuting vector fields $S_h$
and $T_h$ on $M\setminus \Sigma_{h}$ such that 
\begin{enumerate}
\item The frame $\{S_h,T_h\}$ is a parallel  orthonormal frame with respect to the metric $R_{h}$ for the restriction of the tangent bundle $TM$ to the complement $M\setminus \Sigma_{h}$  of the set of cone points;
\item the vector field $S_{h}$ is tangent to the horizontal foliation $\mathcal  F_{h}$, the vector field $T_{h}$
 is tangent to the vertical foliation $\mathcal  F_{-h}$ on $M\setminus \Sigma_{h}$ \cite{F97}, \cite{F07}. 
 \end{enumerate}
 In the following  we will often drop the dependence of the vector fields $S_{h}$, $T_{h}$ on the Abelian differential in order to simplify the notation. The symbols
$\mathcal  L _{S}$, $\mathcal  L _{T}$ denote the Lie derivatives, and $\imath_S$, $\imath_T$ the contraction operators with respect to the vector field  $S$, $T$ on $M\setminus \Sigma_h$. We have:
\begin{enumerate}
\item $\mathcal  L _{S} \omega_{h} = \mathcal  L_{T}\omega_{h} =0$ on $M\setminus \Sigma_{h}$ , that is, the 
area form $\omega_{h}$ is invariant with respect to the flows generated by $S$ and $T$;
\item $\imath_{S} \omega_{h}= \re h$ and $\imath_{T} \omega_{h}= \im h$, hence
the $1$-forms $\eta_{S} :=\imath_{S} \omega_{h}$,  $\eta_{T} :=-\imath_{T} \omega_{h}$ are smooth 
and closed on $M$ and $\omega_{h}= \eta_{T}\wedge \eta_{S}$.
\end{enumerate}
It follows from the area-preserving property $(1)$ that the vector field $S$, $T$ are anti-symmetric
as densely defined operators on $L^{2}_{h}(M)$, that is, for all functions $u$, $v \in C_0^{\infty} (M\setminus\Sigma_h)$,  (see \cite{F97}, $(2.5)$),
\begin{equation}
\label{eq:antisymm}
\< Su ,v\>_{h} = -\< u ,Sv\>_{h}\,\,, \quad \text{ respectively } \,\, \< Tu ,v\>_{h} =-\< u ,Tv\>_{h} \,\,.
\end{equation}
In fact, by Nelson's criterion~\cite{Ne59}, Lemma 3.10, the anti-symmetric operators $S$, $T$ are {\it essentially skew-adjoint} on the Hilbert space $L^{2}_{h}(M)$.

\smallskip
\noindent The {\it weighted Sobolev norms} $\vert \cdot\vert_{k}$, with integer exponent $k>0$, are the euclidean norms, introduced in \cite{F97}, induced by the hermitian product defined as follows: for all 
functions $u$, $v\in L^{2}_{h}(M)$,
\begin{equation}
 \label{eq:knorm}
 \< u,v \>_{k} :=   \frac{1}{2}\sum_{i+j\leq k}\<S^{i}T^{j}u, S^{i}T^{j}v\>_{h} + 
 \<T^{i}S^{j}u, T^{i}S^{j}v\>_{h}\,.
 \end{equation}
The  {\it weighted Sobolev norms }with integer exponent $-k<0$ are defined to be the dual norms.
The {\it weighted Sobolev space }$H^{k}_h(M)$, with integer exponent $k\in\Z$, is the 
Hilbert space obtained as the completion with respect to the norm $\vert \cdot  \vert _{k}$ of the 
maximal {\it common invariant domain}
 \begin{equation}
 \label{eq:cid}
 H^{\infty}_h(M):= \bigcap_{i,j\in \N}  D( \bar S^i \bar T^j) \cap D( \bar T^i \bar S^j)\,.
 \end{equation}
 of the closures $\bar S$, $\bar T$ of the essentially skew-adjoint operators $S$,  $T$ on $L^2_h(M)$.
 The weighted Sobolev space $H^{-k}_h(M)$ is isomorphic to  the dual space of the Hilbert space $H^{k}_h(M)$, for all $k\in \Z$. 
 
 \smallskip
 \noindent  Since the vector fields $S$, $T$ commute  as operators on $C_0^\infty(M\setminus\Sigma_{h})$, the following weak commutation identity holds on $M$. 
  \begin{lemma} 
   \label{lemma:commut}
  (\cite{F97}, Lemma 3.1)  
  For all functions $u$,$v  \in H^{1}_{h}(M)$,
 \begin{equation}
 \label{eq:commut}
\<Su,Tv\>_{h} =  \<Tu,Sv\>_{h}\,\,.
 \end{equation}
  \end{lemma}
  \smallskip
  \noindent By the anti-symmetry property~\eqref{eq:antisymm} and the commutativity property 
  \eqref{eq:commut},  the frame $\{S,T\}$ yields an essentially skew-adjoint action of the Lie 
  algebra $\R^{2}$ on the Hilbert space $L^{2}_{h}(M)$ with common domain $H^{1}_{h}(M)$.  
 
  If $\Sigma_{h}\not=\emptyset$, the (flat) Riemannian manifold $(M\setminus\Sigma_{h}, R_{h})$ 
  is not complete, hence its Laplacian $\Delta_h$ is not essentially self-adjoint on $C_{0}^{\infty}(M \setminus\Sigma_{h})$. 
  By a theorem of Nelson~\cite{Ne59}, \S 9, this is equivalent to the non-integrability of the action of~$\R^{2}$ 
  as a Lie algebra (to an action of $\R^{2}$ as a Lie group). 
 
 \medskip
 \noindent  Following \cite{F97}, the Fourier analysis on the flat surface $(M,h)$ will be based 
 on a canonical self-adjoint extension $\Delta_h^F$ of the Laplacian $\Delta_h$,  called  the 
 {\it Friedrichs extension},  which is uniquely determined by  the {\it Dirichlet hermitian form }$\mathcal  Q:H^{1}_{h}(M)\times H^{1}_{h}(M) \to \C$. We recall that, for all $u$, $v \in H^{1}_{h}(M)$,
 \begin{equation}
 \label{eq:Dirichlet}
 \mathcal  Q(u,v) := \<Su,Sv\>_{h} \,+\, \<Tu,Tv\>_{h} \,\,. 
 \end{equation}
   \begin{theorem} (\cite{F97}, Th. 2.3)
   \label{thm:Dirichlet}
  The hermitian form $\mathcal  Q$ on $L^{2}_{h}(M)$ has the following spectral properties:
  \begin{enumerate}
  \item $\mathcal  Q$ is positive semi-definite and the set $\text{\rm EV}(\mathcal  Q)$ of its eigenvalues is a  
  discrete subset of $[0,+\infty)$;
  \item Each eigenvalue has finite multiplicity, in particular $0\in \text{\rm EV}(\mathcal  Q)$ is simple and the kernel of $\mathcal  Q$ consists only of constant functions;
  \item The space $L^{2}_{h}(M)$ splits as the orthogonal sum of the eigenspaces. In addition,
  all eigenfunctions are $C^{\infty}$ (real analytic) on $M$.
   \end{enumerate} 
  \end {theorem} 
  \noindent The {\it Weyl asymptotics }holds for the eigenvalue  spectrum of the Dirichlet form . For any $\Lambda>0$, let $N_{h}(\Lambda):=\text{\rm card} \{ \lambda \in \text{\rm EV}(\mathcal  Q)\,/\, \lambda \leq \Lambda \}$, where each eigenvalue $\lambda \in \text{\rm EV}(\mathcal  Q)$ is counted according to its multiplicity. 
  \begin{theorem} (\cite{F97}, Th. 2.5) 
    \label{thm:Weyl}
  There exists a constant $C>0$ such that
   \begin{equation}
 \label{eq:Weyl}
\lim_{\Lambda\to + \infty} \,\frac {N_{h}(\Lambda)}{\Lambda} \,\, = \,\, \text{\rm vol}(M,R_{h})\,\,. 
 \end{equation}
 \end {theorem} 
 \medskip
 \noindent Let $\partial^{\pm}_{h}:=S_{h}\pm \imath\, T_{h}$  (with $\imath =\sqrt{-1}$)  be the {\it Cauchy-Riemann 
 operators }induced by the holomorphic Abelian differential $h$ on $M$, introduced in~\cite{F97}, \S 3.  Let $\mathcal  M^{\pm}_{h}\subset L^{2}_{h}(M)$ be the subspaces of meromorphic, respectively anti-meromorphic functions (with poles at $\Sigma_{h}$). By the Riemann-Roch 
 theorem, the subspaces $\mathcal  M^{\pm}_{h}$ have  the same complex dimension equal to 
 the genus $g\geq 1$ of the Riemann surface $M$. In addition,  $\mathcal  M^{+}_{h}\cap 
 \mathcal  M^{-}_{h}=\C$, hence 
  \begin{equation}
  \label{eq:H}
H_{h}:= \left(\mathcal  M^{+}_{h}\right)^{\perp} \oplus  \left(\mathcal  M^{-}_{h}\right)^{\perp} = 
 \{ u\in L^{2}_{h}(M)\,\vert \,  \int_{M} u\, \omega_{h} \,=\,0\,\}\,\,.
 \end{equation}
 Let $H^{1}_{h}: = H_{h} \cap H^{1}_{h}(M)$. By Theorem~\ref{thm:Dirichlet}, the restriction of the
 hermitian form to $H^{1}_{h}$ is positive definite, hence it induces a norm. By the Poincar\'e 
 inequality (see \cite{F97}, Lemma 2.2 or \cite{F02}, Lemma 6.9), the Hilbert space $(H^{1}_{h}, 
 \mathcal  Q)$ is isomorphic to the Hilbert space $(H^{1}_{h}, \<\cdot,\cdot\>_{1})$.
 
 \begin{proposition} (\cite{F97}, Prop. 3.2) 
 \label{prop:CR}
  The Cauchy-Riemann operators $\partial^{\pm}_{h}$ are closable operators on the common domain
 $C^{\infty}_{0}(M\setminus\Sigma_{h}) \subset  L^{2}_{h}(M)$ and their closures (denote by the same symbols) have the following properties:
 
\begin{enumerate}
\item the domains $D(\partial^{\pm}_{h}) = H^{1}_{h}(M)$ and the kernels $N(\partial^{\pm}_{h}) = \C$;
\item the ranges $R_h^{\pm} :=\text{\rm Ran}(\partial^{\pm}_{h}) =   \left(\mathcal  M^{\mp}_{h}\right)^{\perp} $
are closed in  $L^{2}_{h}(M)$;
\item the operators  $\partial^{\pm}_{h}: (H^{1}_{h},  \mathcal  Q) \to (R^{\pm}, \<\cdot,\cdot\>_{h})$ are isometric.
\end{enumerate}
\end{proposition} 
\medskip
\noindent Let $\mathcal  E= \{ e_{n} \,\vert \, n\in \N \}\subset H^{1}_{h}(M)\cap C^{\infty}(M)$ be an orthonormal basis of the Hilbert space $L^{2}_{h}(M)$ of eigenfunctions of the Dirichlet form~\eqref{eq:Dirichlet}  and let $\lambda:\N\to \R^{+}\cup \{0\}$ be the corresponding sequence of eigenvalues:
\begin{equation}
\lambda_{n}:= \mathcal  Q(e_{n}, e_{n}) \,, \quad \text{\rm  for each }\,n \in \N\,.
\end{equation}
We then recall the definition of the {\it Friedrichs (fractional) weighted Sobolev norms and spaces} introduced in  \cite{F07}, \S 2.2. 

\begin{definition}
\begin{enumerate}
\item[(i)] The \emph{Friedrichs (fractional) weighted Sobolev norm} $\Vert \cdot \Vert_{s}$ of order $s\geq 0$ is the norm induced by the hermitian product defined as follows: 
for all $u$, $v\in L^{2}_{h}(M)$,
\begin{equation}
 \label{eq:extsnorm}
 (u,v) _{s} :=   \sum_{n\in \N} (1+\lambda_{n})^{s}\,
 \<u,e_{n}\>_{h}  \,  \<e_{n},v\>_{h} \,;
 \end{equation}
\item[(ii)]  the \emph{Friedrichs weighted Sobolev space} $\bar H_{h}^{s} (M)$ of order $s\geq 0$ is the Hilbert space
 \begin{equation}
\label{eq:Sobbar}
\bar H_{h}^{s} (M):= \{ u\in L^{2}_{h}(M)\,/\,  \sum_{n\in \N} (1+\lambda_{n})^{s}
\vert \<u,e_{n}\>_{h}\vert^{2} < +\infty\,\} 
\end{equation}
endowed with the hermitian product given by \eqref{eq:extsnorm}\,;
\item[(iii)]
the \emph{Friedrichs weighted Sobolev space} $\bar H_{h}^{-s} (M)$ of order $-s<0$
is  the dual space of the Hilbert space $\bar H_{h}^{s} (M)$. 
\end{enumerate}
 \end{definition}
 \noindent  As stated in \cite{F07}, Lemma 2.6, the family of Friedrichs (fractional) weighted Sobolev spaces is a holomorphic interpolation family in the sense of Lions-Magenes  \cite{LiMa}, Chap. 1, endowed with the canonical interpolation norm.

\medskip
\noindent The family $\{H^s_h(M)\}_{s\in \R}$ of {\it fractional weighted Sobolev spaces} will be 
defined as follows. Let $[s]\in \N$ denote the {\it integer part} and $\{s\} \in [0,1)$ the {\it fractional 
part }of any real number $s\geq 0$. 

\begin{definition}
\label{def:snorm}
(\cite{F07}, Def. 2.7)
\begin{enumerate}
\item[(i)] The \emph{fractional weighted Sobolev norm} $\vert \cdot  \vert _{s}$ of order $s\geq 0$ is the euclidean norm induced by the hermitian product defined as follows: for all functions $u$, $v\in H^{\infty}_{h}(M)$,
\begin{equation}
 \label{eq:snorm}
 \< u,v \>_{s} :=   \frac{1}{2}\sum_{i+j\leq [s]} (S^{i}T^{j}u, S^{i}T^{j}v)_{\{s\}} + 
 (T^{i}S^{j}u, T^{i}S^{j}v)_{\{s\}}\,.
 \end{equation}
\item[(ii)] The  \emph{fractional weighted Sobolev norm }$\vert \cdot  \vert _{-s}$ of order $-s<0$ is defined as the dual norm of the weighted Sobolev norm $\vert \cdot  \vert _{s}$. 
\item[(iii)] The \emph{fractional weighted Sobolev space }$H^{s}_h(M)$ of order $s\in \R$ is defined as the completion with respect to the norm $\vert \cdot  \vert _{s}$ of the maximal common invariant domain $H^{\infty}_h(M)$.  
\end{enumerate}
\end{definition}
\noindent It can be proved that the weighted Sobolev space $H^{-s}_h(M)$ is isomorphic to the dual space of the Hilbert space $H^{s}_h(M)$, for all $s\in \R$. 

\smallskip
\noindent The definition of the fractional weighted Sobolev norms is  motivated by the following
basic result.

\begin{lemma} 
\label{lemma:CR_s}
(\cite{F07}, Lemma 2.9)
For all~$s\geq 0$, the restrictions of the Cauchy-Riemann operators 
$\partial^{\pm}_{h}:H^1_h(M) \to L^2_h(M)$ to the subspaces $H^{s+1}_h(M)\subset H^1_h(M)$ 
yield bounded operators 
$$
\partial^{\pm}_{s}:H_{h}^{s+1} (M) \to H_{h}^{s} (M)
$$
(which do not extend 
to operators  $\bar H_{h}^{s+1} (M) \to \bar H_{h}^{s} (M)$ unless $M$ is the torus). 
On the other hand, the Laplace operator 
\begin{equation}
\Delta_{h} =\partial^{+}_{h} \partial^{-}_{h} =\partial^{-}_{h} \partial^{+}_{h}: 
H^{2}_{h}(M) \to L^{2}_{h}(M)
\end{equation} 
yields a bounded operator $\bar \Delta_{s} :\bar H_{h}^{s+2}  (M) \to \bar H_{h}^{s} (M)$, defined
as the restriction of the Friedrichs extension $\Delta_h^F: \bar H_{h}^{2}  (M) \to L^2_h(M)$.
\end{lemma}
\noindent We do not know whether the fractional weighted Sobolev spaces form a holomorphic interpolation family. However, the fractional weighted Sobolev
norms do satisfy interpolation inequalities (see \cite{F07}, Lemma 2.10 and Corollary 2.26). 

\smallskip
\noindent A detailed comparison between Friedrichs weighted Sobolev norms and  weighted Sobolev norms and the corresponding weighted Sobolev spaces 
 is carried out in \cite{F07}, \S 2.  In particular, we have the following result.
 
\noindent Let $H^{s}(M)$, $s\in \R$, denote a family of standard Sobolev spaces on the compact 
manifold $M$ (defined with respect to a Riemannian metric).  

\begin{lemma}  (\cite{F07}, Lemma 2.11)
\label{lemma:comparison}
The following continuous embedding and isomorphisms of Banach spaces hold:
\begin{enumerate}
\item $ \,\, H^{s}(M) \,\, \subset   \,\, H_{q}^{s}(M) \,\, \equiv \,\,  \bar H_{q}^{s}(M) \,,  
\quad\text{for }0\leq s<1$;
\item $\,\,H^{s}(M) \,\, \equiv   \,\, H_{q}^{s}(M) \,\, \equiv \,\,  \bar H_{q}^{s}(M)\,,
\quad\text{for }s=1$;
\item $\,\,H_{q}^{s}(M) \,\, \subset  \,\, \bar H_{q}^{s}(M) \,\, \subset  \,\,  H^{s}(M)\,, 
\quad\text{for }s >1$.
\end{enumerate}
For $s \in [0,1]$, the space $H^{s}(M)$ is dense in $H_{q}^{s}(M)$ and, for $s >1$, the closure 
of $H_{q}^{s}(M)$ in $\bar H_{q}^{s}(M)$ or $H^s(M)$ has finite codimension.   
\end{lemma}
\noindent We also have the following a sharp version of Lemma 4.2 of \cite{F97}:
 \begin{theorem} 
 \label{lemma:equivnorms} (\cite{F07}, Lemma 2.5 and Corollary 2.25 )
 For each $k\in \Z^{+}$ there exists a constant~$C_{k}>1$ such that, for any holomorphic  Abelian differential $h$ on $M$ and for all  $u\in H^{k}_{h}(M)$,
 \begin{equation}
 \label{eq:equivnormsone}
 C_{k}^{-1}\, \vert u \vert_{k}\,\, \leq \,\, \Vert u \Vert_{k} \,\, \leq \,\, C_{k} \, \vert u \vert_{k}\,.
 \end{equation}
 For any $0<r<s$ there exists constants $C_r>0$ and $C_{r,s} >0$ such
  that, for all $u\in H^s_h(M)$, the following inequalities hold:
 \begin{equation}
 \label{eq:equivnormstwo}
C_r^{-1}  \Vert u \Vert_r  \leq  \vert u\vert_r \leq C_{r,s} \, \Vert u \Vert_s\,.
 \end{equation}
\end{theorem}

 \section{ The twisted Beurling tranform}
 \label{sec:Beurling}

\noindent  For every $\sigma \in \R$ and for every Abelian differential $h$, we introduce a family of  partial isometries
$U_{h,\sigma}$  (of Beurling transform type), defined on a finite codimensional subspace of $L^2_h(M):=L^2(M,\omega_h)$, which 
generalized the partial isometry $U_q=U_{h,0}$ (for $q=h^2$) first introduced in \cite{F97}, \S 3, in the study of the cohomological equation 
for translation flows. 

\noindent  The partial isometry $U_{h,\sigma}$ is extended in an arbitrary way to a unitary operator $U_{J,\sigma}$ on the whole space $L^2_h(M)$.  Resolvent estimates for $U_{J,\sigma}$ will appear to be related with {\it a priori }estimates 
for  the twisted cohomological equations for translation flows on $(M, h)$.  Consequently, we derive our results on twisted cohomological
equation from basic estimates on the limiting behavior as $z\to \partial D$ of the resolvent  ${\mathcal R}_U(z):=(U-zI)^{-1}$, defined on
the unit disc $D \subset \C$, of a unitary operator 
$U$ on a general Hilbert space. Such estimates, established in \cite{F97},  are based on fundamental facts of classical harmonic analysis, 
in particular of Fatou's theory on the boundary behavior of holomorphic functions. The results obtained are then specialized to the case of the 
unitary  operator $U:=U_{J,\sigma}$. 
 
\smallskip     
\noindent Let $h$ be a holomorphic Abelian differential on a Riemann 
surface $M$ of genus $g\geq 2$.  Let $\{S,T\}$ be the orthonormal frame for 
$TM$  on $M\setminus\Sigma$ introduced in \S \ref{sec:analysis}. We recall that the $1$-forms 
$\eta_S=\imath_S\omega_h$ and $\eta_T=-\imath_T\omega_h$ are 
closed and describe the horizontal, resp. vertical, foliation of $\omega$ on $M$. It 
is possible to associate to $h$ a one-parameter family of measured  foliations parametrized 
by $\theta\in \T := \R/ 2\pi \Z$ in the following way: let $h_{\theta}:=e^{-\imath\theta} \,h$ and let 
${\mathcal F}_{\theta}$ be the 
horizontal foliation of the Abelian differential $h_{\theta}$,
i.e. the foliation defined by the closed $1$-form 
$$
\im h_{\theta} =\{e^{-\imath\theta}(\eta_T+\imath\eta_S)-
e^{\imath\theta}(\eta_T-\imath\eta_S)\}/2\imath\,\,.$$
The foliation ${\mathcal F}_{\theta}$ can also be obtained by integrating
the dual vector field 
\begin{equation}
\label{eq:Stheta}
S_{\theta}:= (\cos \theta) S + (\sin\theta) T =\{e^{-\imath\theta}(S+\imath\,T)+e^{\imath\theta}(S-\imath\,T)\}/2\,\,, 
\end{equation}
which corresponds to the rotation of the vector field $S$ by an angle 
$\theta \in \T$ in the positive direction.

\smallskip
\noindent  In the following we will denote by $\partial^{\pm}_h$ the {\it Cauchy-Riemann }operators $S\pm \imath\,T$
respectively. The {\it twisted Cauchy-Riemann }operators  
$$
\partial^{\pm}_{h, \sigma} :=  (S+ \imath\sigma) \pm  \imath T  = \partial^{\pm}_h +  \imath\sigma\,, 
$$
will play a crucial role.  For all $\theta \in \T$, let $\sigma_\theta:= \sigma \cos\theta$. We remark that, for every 
$\theta\in \T$, we have
$$
S_{\theta} +  \imath\sigma_\theta:=\{e^{-\imath\theta} \partial^{+}_{h, \sigma} +e^{\imath\theta}\partial^{-}_{h, \sigma}\}/2\,\,, 
$$
hence we have the formal factorization
\begin{equation}
\label{eq:for_fact}
\begin{aligned}
S_{\theta} +  \imath\sigma_\theta &={{e^{-\imath\theta}}\over 2}\,\Bigl((\partial^+_{h, \sigma})
(\partial^-_{h, \sigma})^{-1}+e^{2\imath\theta}\Bigr) \,\partial^{-}_{h, \sigma}
\\ &={{e^{\imath\theta}}\over 2}\,\Bigl((\partial^-_{h, \sigma})
(\partial^+_{h, \sigma})^{-1}+e^{-2\imath\theta}\Bigr) \,\partial^{+}_{h, \sigma}\,\,.
\end{aligned}
\end{equation}
Let $Q_{h, \sigma}$ denote the bilinear form  defined, for all $u, v\in H^1_h(M)$, as follows: 
$$
Q_{h, \sigma} (u,v) := \<(S+ \imath\sigma) u, (S+ \imath\sigma)v\>_h + \<Tu, Tv\>_h\,.
$$
Let $K_{h, \sigma}\subset H^1_h(M) \cap C^\infty(M\setminus \Sigma)$ denote the finite-dimensional subspace
\begin{equation}
\label{eq:CRtwist_ker}
K_{h, \sigma} := \{ u \in H^1_h(M) \cap C^\infty(M\setminus \Sigma)\vert  (S+\imath \sigma)u= Tu=0\}\,.
\end{equation}

\begin{lemma}
\label{lemma:form_comparison}
 The twisted bilinear form $Q_{h, \sigma}$ induces a norm on $K^\perp_{h, \sigma} \cap H^1_h(M)$. 
In fact, for all $\sigma \in \R$, there exists a constant $C_{h} >1$ such that, for all  $u\in K^\perp_{h, \sigma} \cap H^1_h(M)$, 
$$
C_{h}^{-1} Q_{h,0} (u,u) \leq Q_{h, \sigma} (u,u) \leq  C_{h} (1+\sigma^2) \Big ( Q_{h,0} (u,u) +\vert \int_M u \omega_h \vert \Big )\,.
$$
\end{lemma} 
\begin{proof}  Since translation flows are area-preserving, hence symmetric on their common domain,  we have, for all $u\in H^1_h(M)$, 
\begin{equation}
\label{eq:Qsigma}
\begin{aligned}
Q_{h, \sigma} (u,u) :=& \<(S+ \imath\sigma) u, (S+ \imath\sigma)u\>_h  + \Vert Tu \Vert^2_{L^2_h(M)}    \\ = &
 \Vert Su \Vert^2_{L^2_h(M)} +  2  \imath\sigma \<u,Su\>_h  + \sigma^2 \Vert u\Vert^2_{L^2_h(M)}  
 + \Vert Tu \Vert^2_{L^2_h(M)} \,.
\end{aligned}
\end{equation}
By the Cauchy-Schwarz inequality, we have
$$
\vert \<u,Su\>_h \vert \leq   \Vert u\Vert_{L^2_h(M)}  \Vert Su\Vert_{L^2_h(M)}   \leq \frac{  \Vert Su\Vert_{L^2_h(M)} + 
 \Vert u\Vert^2_{L^2_h(M)}}{2}\,.
$$
It follows that
$$
  \Vert Su \Vert^2_{L^2_h(M)} +  2  \imath\sigma \<u,Su\>_h  \leq   (1+ \vert \sigma\vert) \Vert Su \Vert^2_{L^2_h(M)}  + \vert \sigma\vert \Vert u \Vert^2_{L^2_h(M)} \,, 
$$
hence  we derive that 
$$
Q_{h, \sigma} (u,u) \leq     (1+ \vert \sigma\vert) Q_{h,0} (u,u)  +  (\sigma^2 + \vert \sigma\vert) \Vert u \Vert^2_{L^2_h(M)}  \,.
$$
By the Poincar\'e inequality there exists a constant $C_h>0$ such that, for all $u\in  H^1_h(M)$,  we have
$$
\begin{aligned}
Q_{h, \sigma} (u,u) &\leq (1+ \vert \sigma\vert) Q_{h, 0} (u,u)  +  (\sigma^2 + \vert \sigma\vert) \Vert u \Vert^2_{L^2_h(M)}  \\ &\leq 
[(1+ \vert \sigma\vert) +  C  (\sigma^2 + \vert \sigma\vert)]Q_{h, 0} (u,u) +  (\sigma^2 + \vert \sigma\vert)\vert \int_M u \omega_h\vert\,.
\end{aligned}
$$
The upper bound in the statement is therefore proved.

\noindent  To prove the lower bound,  we proceed as follows. By the definition of $Q_{h,\sigma}$, for the splitting $u = v + \bar u  \in \C^\perp \oplus^\perp \C \subset 
H^1_h(M)$ we have
\begin{equation}
\label{eq:Q_lambda_split}
Q_{h, \sigma} (u,u) = Q_{h, \sigma} (v,v) + \sigma^2 \bar u^2\,,
\end{equation}
hence without loss of generality we can reduce the argument to functions $u\in H^1_h(M)$ of
zero average.  By the compact embedding $H^1_h(M) \to L^2(M)$ we derive that there exists
a constant $c_h>0$ such that, for all $u\in H^1_h(M)$ we have 
$$
\Vert u \Vert^2_{L^2_h(M)} \geq c^2_h Q_{h,0} (u,u) \geq  c^2_h \Vert Su \Vert^2_{L^2_h(M)} \,.
$$
It follows then by formula \eqref{eq:Qsigma}  that, for $\vert \sigma \vert \geq 2 c_h^{-1}$ we have
$$
Q_{h, \sigma} (u,u) \geq Q_{h, 0} (u,u) + \vert \sigma\vert \Vert u \Vert_{L^2_h(M)} ( \vert \sigma\vert  
\Vert u \Vert_{L^2_h(M)}  - 2 \Vert Su \Vert_{L^2_h(M)} ) \geq Q_{h, 0} (u,u)\,.
$$
It remains to prove the bound for $\vert \sigma \vert \leq 2 c_h^{-1}$. 
Let us then assume by contradiction that for all $n\in \N$ there exist a bounded sequence $(\sigma_n)$ and a sequence   $u_n\in K_{h, \sigma_n}^\perp \subset H^1_h(M)$ of zero average  such that
$$
Q_{h, 0} (u_n,u_n) \geq n Q_{h, \sigma_n} (u_n,u_n) \,.
$$
After normalizing, it is not restrictive to assume that $Q_{h, 0} (u_n,u_n)=1$, for all $n\in \N$, hence $Q_{h, \sigma_n} (u_n,u_n)\to 0$.  
By the Poincar\'e inequality, it follows that after passing to a subsequence we can assume that $u_n \to u$ in  $L^2(M)$ and $u$ has zero average, as well as that $\sigma_n \to \sigma \in \R$.  Let $\Phi_S^\R$ and
$\Phi_R^\R$ denote, respectively, the horizontal and the vertical flow. By assumption, since 
$$
\begin{aligned}
\Vert  e^{\imath \sigma_n t} u_n\circ \Phi_S^t - u_n \Vert_{L^2_h(M)} &= \Vert  \int_0^t (S+ \imath \sigma_n) u_n\circ \Phi_S^s ds\Vert_{L^2_h(M)}  \\&
\leq  \int_0^t \Vert (S+ \imath \sigma_n)u_n\circ \Phi_S^s \Vert_{L^2_h(M)} ds \leq t Q^{1/2}_{h, \sigma_n}(u_n,u_n) \to 0\; \\
\Vert   u_n\circ \Phi_T^t - u_n \Vert_{L^2_h(M)} &= \Vert  \int_0^t Tu_n\circ \Phi_T^s ds \Vert_{L^2_h(M)} \\ &\leq  \int_0^t 
\Vert Tu_n\circ \Phi_T^s \Vert_{L^2_h(M)} ds \leq t Q^{1/2}_{h, \sigma_n}(u_n,u_n) \to 0\,,
\end{aligned}
$$
it follows that the limit function $u\in K_{h, \sigma}^\perp \subset L^2_h(M)$ is a zero-average eigenfunction of eigenvalue $- \imath\sigma$ for the flow $\Phi_S^\R$ and it is
invariant for the flow $\Phi_T^\R$.  It follows in particular that $u\in C^\infty(M\setminus \Sigma)\cap H^1_h(M)$ which implies that $u$ is constant on all minimal
components of the flow $\Phi_T^\R$ and $\Phi_T^\R$-invariant on cylindrical component.  In particular $u\in K_{h, \sigma}$, hence $u=0$.  However, from $u_n\to 0$ in
$L^2_h(M)$ and $\sigma_n \to \sigma$, from the identity in formula~\eqref{eq:Qsigma} we then derive 
$$
0= \lim_{n\to \infty} Q_{h, \sigma_n} (u_n,u_n) =   \lim_{n\to \infty} Q_{h, 0}(u_n,u_n) =1\,,
$$
a contradiction. We have thus proved that there exists $C_{h} >1$ such that, for all $u\in K_{h, \sigma}^\perp\cap H^1_h(M)$, 
$$
C_h^{-1} Q_{h, 0}(u,u)  \leq Q_{h, \sigma} (u,u)\,.
$$

\end{proof} 

\noindent The twisted Cauchy-Riemann operators $\partial^{\pm}_{h, \sigma}$ on 
$L^2_h(M)$ will be described in the following Proposition. 
\begin{proposition} \label{prop:CR_twisted} The Cauchy-Riemann operators $\partial^{\pm}_{h, \sigma}$
are closable operators on $C^{\infty}_0(M\setminus\Sigma)\subset L^2_h(M)$ 
and their closures (denoted by the same symbols) have the following properties:
\begin{enumerate}
\item[(i)] $D(\partial^{\pm}_{h, \sigma})=H^1_h(M)$ and $N(\partial^{\pm}_{h, \sigma})=K_{h, \sigma}\subset H^1_h(M)$.
\item[(ii)] The kernels  ${\mathcal M}^\pm_{\Sigma}(\sigma) \subset L^2_h(M)$ of the adjoint operators $(\partial^{\mp}_{h, \sigma})^*$ 
have finite dimensions  $d^\pm(\sigma)$ and there exists $d_h \in \N$ such that
$$
d^+ (\sigma) = d^- (-\sigma) = d_h\,, \quad \text{ for all } \sigma \in \R.
$$
\item[(iii)] the adjoints $(\partial^\pm_{h, \sigma})^*$ of $\partial^\pm_{h, \sigma}$ are extensions of $-\partial^\mp_{h, \sigma}$, and we have closed ranges
$$R^{\pm}_{h, \sigma}:=\hbox{Ran}\,(\partial^\pm_{h, \sigma})=[{\mathcal M}^\mp_{\Sigma}(\sigma)] ^{\perp}\,.$$
\item[(iv)] The operators $\partial^{\pm}_{h, \sigma} :(K^\perp_{h, \sigma} \cap H^1_h(M), Q_{h, \sigma})\to 
(R^{\pm}_{h, \sigma},(\,\cdot\,,\,\cdot\,)_h)$ are isometric. 
\end{enumerate}
\end{proposition}
\begin{proof}
 If $u$, $v\in H^1_h(M)$, Lemma \ref{lemma:commut} implies the 
following identity:
\begin{equation}
\label{eq:isometry}
\begin{aligned}
\<\partial^{\pm}_{h, \sigma} u,&\partial^{\pm}_{h, \sigma} v\>_h=
\<(S+\imath \sigma )u, (S+\imath \sigma )v\>_h+ \<Tu,Tv\>_h  \\ 
 &\pm \imath \bigl(\<Tu,(S+\imath \sigma )v\>_h-\<(S+\imath \sigma )u,Tv\>_h\bigr)= Q_{h, \sigma}
(u,v)\,\,.
\end{aligned} 
\end{equation}
It follows immediately that the operators $\partial^{\pm}_{h, \sigma}$ are closed with domain $D(\partial^{\pm}_{h, \sigma})=H^1_h(M)$
and that their kernels are both equal to $K_{h, \sigma} \subset H^1_h(M)$. 

\noindent  Since, for all $u,v\in H^1_h(M)$, we have
$$
\begin{aligned}
\<\partial^\pm_{h, \sigma} u, v\>_h&= \<[(S+\imath \sigma) \pm \imath T] u, v\>_h  \\ &= -\< u, [(S+\imath \sigma) \mp  \imath T] v\>_h =
-\< u, \partial^\mp_{h, \sigma} v\>_h\,,
\end{aligned}
$$
the adjoint $(\partial^\pm_{h, \sigma})^*$ of $\partial^\pm_{h, \sigma}$ is an extension of $-\partial^\mp_{h, \sigma}$.  By the theory of elliptic
partial differential equations, the distributional kernels  ${\mathcal M}^\pm_{\Sigma}(\sigma)$ of  $\partial^\mp_{h, \sigma}$ in $L^2_h(M)$ are finite dimensional
subspaces of $C^\infty(M\setminus \Sigma)$, which depend continuously on $\sigma \in \R$, hence have constant dimension. Since by complex conjugation we have
$\overline{\partial^+_{h, \sigma}} = \partial^-_{h, -\sigma}$, it follows that $\overline{{\mathcal M}^+_{\Sigma}(\sigma)} ={\mathcal M}^-_{\Sigma}(-\sigma)$,
hence the dimension $d^+(\sigma) = d^-(-\sigma)$ is constant over $\sigma \in \R$.

\smallskip
\noindent  The formula for the range
$R^\pm_{h, \sigma}$ follows from a general fact of Hilbert space theory, as soon as we have proved that the range is closed. It follows from Lemma
\ref{lemma:form_comparison} that $R^\pm_{h, \sigma}$ are closed. In fact, the subspaces $R^\pm_{h, \sigma}$ coincide with the range of restrictions of the
operators $\partial^\pm_{h, \sigma}$ to the subspace $K_{h, \sigma}^\perp \cap H^1_h(M)$. By Lemma~\ref{lemma:form_comparison}  these restrictions
have closed range.

\noindent  Finally, $(iv)$ is a direct consequence of the  identity in formula~\eqref{eq:isometry}.
\end{proof}

\noindent The results just proved in Proposition \ref{prop:CR_twisted}, in particular 
$(iv)$, allow us to give a precise meaning to the formal factorization \eqref{eq:for_fact}, 
by introducing a family of unitary operators on $L^2_h(M)$, which, as 
it will be seen, contains a great deal of information about the 
properties of the differential operator $S_{\theta} + \imath \sigma $ ($\theta\in \T$) defined
in formula \eqref{eq:Stheta}. Let $U_{h, \sigma} :R^-_{h, \sigma} \to R^+_{h, \sigma}$ be defined as
\begin{equation}
\label{eq:partial_iso}
U_{h,\sigma}:=(\partial^+_{h, \sigma})(\partial^-_{h, \sigma})^{-1}\,\,.
\end{equation}
It is an immediate consequence of the assertion $(iv)$ of Proposition \ref{prop:CR_twisted}
that $U_{h,\sigma}$ is a partial isometry. Thus, we extend in the natural way 
the domain of definition of $U_{h,\sigma}$ as follows. Let 
\begin{equation}
\label{eq:finite_iso}
J:{\mathcal M}^+_{\Sigma}(\sigma) \to {\mathcal M}^-_{\Sigma} (\sigma)
\end{equation}
be an isometric operator, with respect to the euclidean structures 
induced on ${\mathcal M}^+_{\Sigma}(\sigma)$ and $ {\mathcal M}^-_{\Sigma}(\sigma)$ 
by the Hilbert space $L^2_h(M)$. The existence of $J$ is a consequence
of the fact that the {\it deficiency subspaces }${\mathcal M}^+_{\Sigma}(\sigma)$ and 
${\mathcal M}^-_{\Sigma}(\sigma)$ are isomorphic finite dimensional vector 
spaces of the same complex dimension (equal to the genus $g$ of the surface $M$). 
In fact, there exists a whole family of operators $J$ as required, 
parametrized by the Lie group $U(g,\C)$. Let $\pi^{\pm}_{h, \sigma}:L^2_h(M)
\to R^{\pm}_{h, \sigma}$ be the orthogonal projections. We recall that $R^\pm_{h, \sigma}$ 
are the orthogonal complements of ${\mathcal M}^\mp_{\Sigma}(\sigma)$ respectively (Proposition \ref{prop:CR_twisted}). 
Once an isometric operator  $J$ as in formula \eqref{eq:finite_iso} is fixed, the partial isometry $U_{h,\sigma}$, 
associated with the 
holomorphic Abelian differential $h$ on $M$ and $\sigma \in \R$ as in formula~\eqref{eq:partial_iso}, will be extended
to a unitary operator $U_{J,\sigma}$ on the whole $L^2_h(M)$ by the formula
$$
U_{J,\sigma} (u):=U_{h,\sigma} \pi^-_{h, \sigma}(u) +J(I-\pi^-_{h, \sigma})(u)\,, \quad \text{ for all }\,\,u\in L_h^2(M)
$$
(the dependence of the unitary operator $U_{J, \sigma}$ on the Abelian differential
is omitted in the notation for convenience).

\noindent The following version of the formal identities \eqref{eq:for_fact}
holds on $H^1_h(M)$:
\begin{equation}
\label{eq:Stheta_id}
S_{\theta} +\imath \sigma_\theta ={{e^{-\imath\theta}}\over 2}\,\Bigl(U_{J, \sigma}+e^{2\imath\theta}\Bigr)
\,\partial^-_{h,\sigma}={{e^{\imath\theta}}\over 2}\,\Bigl(U_{J, \sigma} ^{-1}+e^{-2\imath\theta}\Bigr)
\,\partial^+_{h, \sigma}\,\,.
\end{equation}
{\it A priori }estimates for $S_{\theta}$ are related by formula \eqref{eq:Stheta_id} to 
estimates for the resolvent ${\mathcal R}_{J, \sigma}(z):=(U_{J, \sigma}-zI)^{-1}$ of any of
the operators $U_{J, \sigma}$, as $z\to \T$ non-tangentially. Since these are
unitary operators, their spectrum is contained in the unit circle  $\{z\in \C \vert \, \vert z\vert=1\}$. 
As a consequence, the resolvent ${\mathcal R}_{J, \sigma}(z)$ 
is a well defined operator-valued {\it holomorphic }function on the unit 
disk $D:=\{z\in {\C}\,|\,|z|<1\}$. In addition, by the {\it spectral 
theorem }for unitary operators, it is given by a {\t Cauchy integral }on 
$\partial D$ of the spectral measure associated with $U_{J, \sigma}$. 

\section{Spectral theory of unitary operators}
\label{sec:spectral_unitary}

\noindent  Let $U:{\mathcal H}\to {\mathcal H}$ is {\it any }unitary operator on a (separable) 
Hilbert space $\mathcal H$. By the spectral theorem 
\cite{Yo}, XI.4, its resolvent ${\mathcal R}_U(z):=(U-zI)^{-1}$ 
can be represented as a Cauchy integral of the spectral family, as follows. For any $u$, $v\in {\mathcal H}$,
$$({\mathcal R}_U(z)u,v)_{\mathcal H}=\int_0^{2\pi}
(z-e^{i t})^{-1}\,d(E_U(t)u,v)_{\mathcal H}\,\,, \quad 
 \text{ for all } \,z\in D\,\,,$$
where $(\,\cdot\,,\,\cdot\,)_{\mathcal H}$ denotes the inner product 
in ${\mathcal H}$ and $\{E_U(t)\}_{0\leq t\leq 2\pi}$ denotes the spectral 
family associated with the unitary operator $U$ on $\mathcal H$. Our 
approach from~\cite{F97} is based on the fundamental property of holomorphic functions 
on $D$, which can be represented as Cauchy integrals on $\partial D$ of complex measures, 
of having non-tangential boundary values {\it almost everywhere}.

\noindent  We recall below general results on the boundary behavior of Cauchy integrals of complex 
measures. The modern theory of these singular integrals, based on the
Lebesgue integral, was initiated by P.~Fatou in his thesis \cite{Ft}.
The results gathered  below were originally obtained by F.~Riesz \cite{Rz},
V.~Smirnov \cite{Sm}, G.~H.~Hardy and J.~E.~Littlewood \cite{HL}. The arguments,
given in \cite{F97}, \S 3,  follow the approach of A.~Zygmund \cite{Zy}, VII.9, based 
on real variables methods, and is taken from the books of W.~Rudin
\cite{Rd} and E.~M.~Stein and G.~Weiss \cite{SW}. 

\noindent  Let $\mu$ be a complex Borel measure (of finite total mass)
on $\partial D$. The Cauchy integral of $\mu$ is the holomorphic 
function $I_{\mu}$ on $D$ defined as
\begin{equation}
\label{eq:CI}
I_{\mu}(z):=\int_0^{2\pi}(z-e^{i t})^{-1}\,d\mu (t)\,\,, \quad \text{ for all }
\,\,z\in D\,\,.
\end{equation}

\begin{lemma} 
\label{lemma:CI} (\cite{F97}, Lemma 3.3 A)
The non-tangential limit
$$I_{\mu}(z)\to I_{\mu}^{\ast}(\theta)\,, \quad
\text{ as }\,\,\,\,z\to e^{\imath\theta}\,,$$
exists almost everywhere with respect to the (normalized) Lebesgue 
measure $\mathcal L$ on the circle $\T$. In addition, there exists 
a constant $C>0$ such that the following weak type  estimate holds:
$${\mathcal L}\{\theta\in \T\,|\,\,|I_{\mu}^{\ast}(\theta)|>
t \}\leq {C \over {t}}\,|\!|\mu|\!|\,\,, \quad \text{ for all } t>0\,$$
where $|\!|\mu|\!|$ denotes the total mass of the measure $\mu$. 
\end{lemma}

\noindent  Lemma \ref{lemma:CI} is not enough for our purposes, since it gives no information concerning the behavior of Cauchy integrals
\eqref{eq:CI} as the convergence to the limiting boundary values takes place. 
The necessary estimates are given below, following~\cite{F97},  in terms of 
non-tangential maximal functions, the definition of which we recall below following \cite 
{Rd}, \S 11.18. 

\noindent  For $0<\alpha<1$, we define the non-tangential approach region 
$\Omega_{\alpha}$ to be the cone over $D(0,\alpha)$ of vertex $z=1$, that is, 
the union of the disk $D(0,\alpha)$ and the line 
segments from the point $z=1$ to the points of $\Omega_{\alpha}$. 
Rotated copies of $\Omega_{\alpha}$, having vertex at $e^{\imath\theta}$, will be denoted 
by $\Omega_{\alpha}(\theta)$. 

\noindent  For any complex function $\Phi$ 
on the unit disk $D$ and $0<\alpha<1$, its {\it non-tangential maximal function }$N_{\alpha}(\Phi)$ 
is defined on $\T$ as
\begin{equation} 
\label{eq:maximal_funct}
N_{\alpha}(\Phi)(\theta):=\sup\{|\Phi(z)|\,|\,z\in \Omega_{\alpha}
(\theta)\}\,\,.
\end{equation}
\noindent  We would like to complete Lemma \ref{lemma:CI} with estimates on the non-tangential
maximal function $N_{\alpha}(I_{\mu})$ of the Cauchy integral in formula~\eqref{eq:CI}. 
This can be accomplished by a standard argument of basic Hardy space 
theory. 

\noindent  For the convenience of the reader, we will recall the definition 
of Hardy spaces $H^p(D)$ on the unit disk $D$ \cite{Rd}, \S\S 17.6-7. Let 
$\Phi$ be a complex function on $D$. For $0<r<1$, we define the functions 
$\Phi_r$ on $\T$ by the formula
$$\Phi_r(\theta):=\Phi(re^{\imath\theta})\,\,, \quad \text{ for all } \theta \in \T\,,$$
and, for $0<p\leq \infty$, we define
$$|\!|\Phi|\!|_p=\sup \{|\Phi_r|_p\,|\,0\leq r<1\}\,\,,$$
where $|\cdot|_p$ denotes the $L^p$ norm on $\T$ with respect to 
the Lebesgue measure. The {\it Hardy space} $H^p(D)$ is defined to 
be the space of holomorphic functions $\Phi$ on the unit disk $D$ 
such that $|\!|\Phi|\!|_p<\infty$.

\begin{lemma}  
\label{lemma:CIM} (\cite{F97}, Lemma 3.3B)
The holomorphic function $I_\mu$, given
as a Cauchy integral (see formula~\eqref{eq:CI})  of a Borel complex measure 
$\mu$ on $\partial D$, belongs to the Hardy spaces $H^p(D)$, for any $0<p<1$. 
(Consequently, it admits non-tangential limit almost 
everywhere on $\partial D$). In addition, its non-tangential
maximal function $N_{\alpha}(I_{\mu})$ belongs to $L^p(\T,{\mathcal L})$, 
for any $0<p<1$ and for all $\alpha <1$, and there exist constants 
$A_{\alpha}$, $A_{\alpha,p}>0$, with $A_{\alpha,p}\to \infty$ as $p\to 1$,  such that the following 
estimates hold: 
$$\vert N_{\alpha}(I_{\mu}) \vert_p\leq A_{\alpha}|\!|I_{\mu}|\!|_p\leq 
A_{\alpha,p} |\!|\mu|\!|\,\,,$$
where $|\!|\mu|\!|$ denotes the total mass of the measure $\mu$. 
\end{lemma}

\noindent  The general harmonic analysis Lemmas \ref{lemma:CI} and \ref{lemma:CIM}  are
then applied via the spectral theorem to the resolvent of an arbitrary
unitary operator on a Hilbert space. The abstract Hilbert space result 
which we obtain in this way will then be applied to the unitary 
operators $U_{J, \sigma} $, $U_{J, \sigma} ^{-1}$ introduced in \S~\ref{sec:Beurling}.

\begin{corollary}
\label{cor:resolvent} 
(\cite{F97}, Corollary 3.4)
 Let ${\mathcal R}_U(z):{\mathcal H}\to {\mathcal H}$, 
$z\in D$, denote the resolvent of a unitary operator $U:{\mathcal H}\to 
{\mathcal H}$ on a Hilbert space $\mathcal H$. Then, for any $u$, $v\in 
{\mathcal H}$, the holomorphic functions $\Phi(u,v)(\cdot):=
({\mathcal R}(\cdot)u,v)_{\mathcal H}$ belong to the Hardy spaces $H^p(D)$, 
for any $0<p<1$. Consequently, they admit non-tangential limit almost 
everywhere on $\partial D$. Furthermore, their non-tangential maximal 
functions $N_{\alpha}(u,v)$ belong to $L^p( \T,{\mathcal L})$, for any 
$0<p<1$ and for all $\alpha <1$, and there exist constants $A_{\alpha}$, 
$A_{\alpha,p}>0$, with $A_{\alpha,p}\to \infty$ as $p\to 1$ , such that the following estimates hold: 
$$|N_{\alpha}(u,v)|_p\leq A_{\alpha}|\!|\Phi(u,v)|\!|_p\leq 
A_{\alpha,p} |\!|u|\!|_{\mathcal H}\,|\!|v|\!|_{\mathcal H}\,\,,$$
where $|\!|\cdot|\!|_{\mathcal H}$ denotes 
the Hilbert space norm. \end{corollary}

\section{Solutions of the twisted cohomological equation}
\label{sec:STCE}

\noindent In this section we adapt to the twisted cohomological equation the streamlined version 
\cite{F07} of the main argument of  \cite{F97} (Theorem~4.1) given with the goal of establishing the sharpest 
bound on the loss of Sobolev regularity  within the reach of the methods of \cite{F97}. 

\subsection{Distributional solutions}
\label{DS}

We derive results on distributional solutions of the twisted cohomological equation from the harmonic analysis results of \S \ref{sec:spectral_unitary}  about the boundary behavior of the resolvent of a unitary 
operator.

\medskip
\begin{definition} 
\label{def:distsol}
Let $h$ be an Abelian differential and let $\sigma \in \R$. A distribution $u\in \bar H^{-r}_h(M)$ will be called a \emph {(distributional) solution} of the cohomological equation $(S +  \imath\sigma) u =f$ for a given function $f\in  \bar H^{-s}_h(M)$  if 
$$
\<u, (S + \imath\sigma) v\> = - \< f, v\>\, ,\quad \text{ \rm for all } \,\, v\in H^{r+1}_h(M) \cap \bar H^{s}_q(M)\,.
$$
\end{definition}

\smallskip
\noindent  Let $h_\theta = e^{-\imath \theta} h$ be its rotation and let $\sigma_\theta:= \sigma \cos\theta$. Let $\{S_\theta\}$ denote the one-parameter family of rotated vector fields introduced in formula~\eqref{eq:Stheta}: 
$$
S_{\theta}:=\{e^{-\imath\theta}(S+\imath\,T)+e^{\imath\theta}(S-\imath\,T)\}/2\,.
$$
For all $s\in \R$, let $\Cal H ^s_h (M) \subset H^s_{h} (M)$  the subspace of distributions 
vanishing on constant functions.

\begin{theorem}
\label{thm:distsol} Let $h$ be an Abelian differential on $M$ with minimal vertical foliation. 
Let $r>2$ and $p\in (0,1)$ be such that $rp>2$. For any $\sigma\in \R$, there exists a bounded linear operator
$$
\Cal U_\sigma: \Cal H^{-1}_{h}(M) \to L^p\left( \T,   {\bar H}^{-r}_h(M)\right)
$$
such that the following holds. For any $\sigma\in \R$ and any $f\in H^{-1}(M)$ there exists a full measure subset ${\Cal F}_r(\sigma,f)\subset  \T$ such that  $u:=\Cal U_\sigma(f)(\theta)\in {\bar H}^{-r}_h(M)$ is a distributional solution of the cohomological equation $(S_{\theta} +  \imath\sigma_\theta) u= f$, for all $\theta \in  {\Cal F}_r(f,\sigma)$.
In addition, there exists a constant $B_h:=B_h(p,r)>0$ such that, for all $f\in  \Cal H_h^{-1}(M)$, vanishing on constant functions,
$$
 \vert \Cal U_\sigma(f) \vert_{p}:= \left(\int_{ \T} \Vert \Cal U_\sigma(f)(\theta) \Vert_{-r}^p \,d\theta \right)^{1/p}
 \,\, \leq \,\,  B_{h} \, \Vert f \Vert_{-1} \,\,.
 $$
\end{theorem}

\noindent  The above theorem is a consequence of the following estimate:
\begin{lemma}
\label{lemma:aprioribound}
Let $h$ be an Abelian differential on $M$ with minimal vertical foliation. 
Let $r>2$ and $p\in (0,1)$ be such that $rp>2$.  For any $\sigma\in \R$  and any $f\in 
\Cal H_{h}^{-1}(M)$, vanishing on constant functions, there exists a measurable function $A_{h,\sigma}(f):=A_{h,\sigma} (f, p,r) \in L^p( \T, {\mathcal L})$ such that, for all $v\in H_h^{r+1}(M)$ 
we have
\begin{equation}
\label{eq:aprioribound}
\vert \<f,v\> \vert \leq A_{h,\sigma}(f, \theta) \,\Vert (S_{\theta}+  \imath\sigma_\theta) v \Vert_r\,\,.  
\end{equation}
In addition, the following bound for the $L^p$ norm of the function $A_{h,\sigma}(f)$ holds. There exists a constant $B_h:=B_{h}(p,r)>0$ such that, for every $\sigma \in \R$ and for every $f\in H^{-1}_h(M)$, vanishing on constant functions, we have 
\begin{equation}
\label{eq:Lpbound}
\vert A_{h,\sigma}(f) \vert_p \leq B_{h} \, \Vert f \Vert_{-1} \,\,.
\end{equation}

\end{lemma} 
\begin{proof}
We recall the formulas \eqref{eq:Stheta_id}:
\begin{equation}
\label{eq:CEidentity}
S_{\theta} +\imath \sigma_\theta ={{e^{-\imath\theta}}\over 2}\,\Bigl(U_{J, \sigma}+e^{2\imath\theta}\Bigr)
\,\partial^-_{h, \sigma}={{e^{\imath\theta}}\over 2}\,\Bigl(U_{J, \sigma} ^{-1}+e^{-2\imath\theta}\Bigr)
\,\partial^+_{h, \sigma}\,.
\end{equation}
The proof of estimate \eqref{eq:aprioribound} is going to be based on properties of the resolvent 
of the operator $U_{J,\sigma}$. In fact, the proof of \eqref{eq:aprioribound} is based on the results, 
summarized in \S~\ref{sec:spectral_unitary}, concerning the non-tangential boundary behavior 
of the resolvent of a unitary operator on a Hilbert space, applied to the operators $U_{J, \sigma}$,  
$U_{J, \sigma}^{-1}$ on $L^2_h(M)$. The Fourier analysis of \cite{F97}, \S 2, also plays a relevant role 
through Lemma \ref{lemma:equivnorms} and the Weyl's asymptotic formula (Theorem \ref{thm:Weyl}). 

\smallskip
\noindent  Since the vertical foliation of $h$ is minimal, it follows that all $T$-invariant functions
in the space $H^1_h(M)$ are constant, hence the common kernel of the twisted Cauchy-Riemann
operators $K_{h,\sigma} \subset H^1_h(M)$ coincides with the subspace of constant functions. 

\smallskip
\noindent Following \cite{F97}, Prop. 4.6A, or \cite{F02}, Lemma 7.3, we prove that there exists a constant $C_{h}>0$ such that the following holds. 

\noindent For any $\sigma\in \R$ and for any distribution $f\in H^{-1}_h(M)$ there exist  (weak) solutions $F^{\pm}_\sigma \in L^2_h(M)$ of the equations  $\partial_{h,\sigma} ^{\pm} F^{\pm}_\sigma=f$ such that 
\begin{equation}
\label{eq:CRsolbound}
\vert F^{\pm}_\sigma\vert_0 \leq C_{h} \, \Vert f \Vert_{-1}\,.
\end{equation}
In fact, the maps given by
\begin{equation}
\label{eq:CRsol}
 \partial^{\pm}_{h,\sigma} v  \to  - \<f,v\> \, \,, \quad \hbox{ for all }v\in H^1_h(M)\,,
\end{equation}
are bounded linear functionals on the (closed) ranges $R^{\pm}_{h,\sigma}\subset L^2_h(M)$ of the twisted Cauchy-Riemann operator $\partial^{\pm}_{h, \sigma} :H^1_h(M)\to L^2_h(M)$. In fact, the functionals are well-defined since by assumption $K_{h, \sigma}=\C$ and  $f$ vanishes on constant function, and they are bounded since, by Lemma~\ref{lemma:form_comparison} and Proposition~\ref{prop:CR_twisted}, there exists a constant $C_{h}>0$ such that, for any $\sigma \in \R$ and for any $v\in H^1_h(M)$ of zero average,
\begin{equation}
\label{eq:solboundone}
\begin{aligned}
\vert \<f,v\> \vert \,&\leq\, \Vert f\Vert_{-1} \vert v \vert_1  \leq C'_h \Vert f\Vert_{-1} Q_{h,0}(v)  \\ & \leq
 \, C_h \Vert f\Vert_{-1} Q_{h,\sigma}(v)\, = \, C_{h} \, \Vert f\Vert_{-1} \, 
\vert \partial^{\pm}_{h,\sigma} v \vert_0\,\,.
\end{aligned}
\end{equation}
Let $\Phi^{\pm}_\sigma$ be the unique linear extension of the linear map \eqref{eq:CRsol} to $L^2_h(M)$ 
which vanishes on the orthogonal complement of $R^{\pm}_{h,\sigma}$ in $L^2_h(M)$. By \eqref{eq:solboundone}, the functionals $\Phi^{\pm}_\sigma$  are bounded on $L^2_h(M)$ with norm
$$
\Vert \Phi^{\pm}_\sigma \Vert  \leq   C_{h} \, \Vert f\Vert_{-1}\,.
$$
By the Riesz representation theorem, there exist two (unique) functions $F^{\pm}_\sigma\in L^2_h(M)$ such
that
$$
\<v, F^{\pm}_\sigma\>_h  \,=\,   \Phi^{\pm}_\sigma (v)\,, \quad \text{ \rm for all } \, v\in L^2_h(M)\,.
$$
The functions $F^{\pm}_\sigma$ are by construction (weak) solutions of the twisted Cauchy--Riemann equations $\partial_{h,\sigma}^{\pm} F^{\pm}_\sigma=f$ satisfying the required bound \eqref{eq:CRsolbound}.

\smallskip
\noindent The identities \eqref{eq:CEidentity}  immediately imply that
\begin{equation}
\label{eq:keyidentity}
\begin{aligned}
\<{\partial}^{\pm}_{h,\sigma} v,F^{\pm}_\sigma\>_h&=2 e^{\mp i \theta}\, \<{\Cal R}^{\pm}_{J,\sigma}(z)(S_{\theta}+ \imath\sigma_\theta) v,F^{\pm}_\sigma\>_h \\ &-(z+e^{\mp2 \imath\theta})\<{\Cal R}^{\pm}_{J,\sigma}(z){\partial}^{\pm}_h v,F^{\pm}_\sigma\>_h\,,
\end{aligned}
\end{equation}
where ${\Cal R}^{+}_{J,\sigma}(z)$ and ${\Cal R}^{-}_{J,\sigma}(z)$ denote the resolvents of the unitary operators 
$U_{J,\sigma}$ and $U_{J,\sigma}^{-1}$ respectively, which yield holomorphic families of bounded operators on
the unit disk $D\subset \C$.  

\smallskip
\noindent Let $r>2$ and let $p\in (0,1)$ be such that $pr>2$. Let $\Cal E=\{e_k\}_{k\in {\N}}$ be the orthonormal Fourier basis of the Hilbert space $L^2_h(M)$ described in \S \ref{sec:analysis}. By Corollary \ref{cor:resolvent} all holomorphic functions
\begin{equation}
\label{eq:matrixelements}
{\Cal R}^{\pm}_{h,\sigma,k}(z):=\<{\Cal R}^{\pm}_{J,\sigma}(z) e_k,F^{\pm}_\sigma\>_h\,\,,\quad k \in {\N}\,,
\end{equation}
belong to the Hardy space $H^p(D)$, for any $0<p<1$. The corresponding non-tangential maximal functions $N_k^{\pm}$ (over cones of arbitrary fixed aperture $0<\alpha<1$) belong to the space $L^p( \T,{\mathcal L})$ and for all $0<p<1$  there exists a constant $A_{\alpha,p}>0$ such that, for any Abelian differential $h$ on $M$, for every $\sigma\in \R$ and $k\in \N$,  the following inequalities hold:
\begin{equation}
\label{eq:Nkbound}
\vert N_{h,\sigma,k}^{\pm} \vert_p\leq A_{\alpha,p} \vert e_k \vert_0\, \vert F^{\pm} _\sigma\vert_0=
A_{\alpha,p} \, \vert F^{\pm} _\sigma\vert_0 \leq  A_{\alpha,p}\, C_{h} \, \vert f\vert _{-1}  \,\,. 
\end{equation}
Let $\{\lambda_k\}_{k\in {\N}}$ be the sequence of the eigenvalues of the Dirichlet form $\Cal Q:= \Cal Q_{h,0}$ introduced in \S \ref{sec:analysis}. Let $w\in  \bar H^r_h(M)$. We  have
\begin{equation}
\label{eq:Rsum}
\<{\Cal R}^{\pm}_{J,\sigma}(z)w,F^{\pm}_\sigma\>_h = \sum_{k=0}^{\infty}\<w,e_k\>_h\,\,{\Cal R}^{\pm}_{h,\sigma,k}(z) \,\,,
\end{equation}
hence, by the Cauchy-Schwarz inequality,
\begin{equation}
\label{eq:Rbound}
\vert\<{\Cal R}^{\pm}_{J,\sigma}(z)w,F^{\pm}_\sigma\>_h\vert \leq \Bigl( \sum_{k=0}^{\infty}
\frac{\vert {\Cal R}^{\pm}_{h,\sigma,k}(z)\vert^2}{ (1+\lambda_k)^r}\,  \Bigr) ^{1/2}
 \Vert w\Vert_r \,\,,
\end{equation}
Let $N^{\pm}_{h,\sigma}(\theta)$ be the functions defined as 
\begin{equation}
\label{eq:Ndef}
N^{\pm}_{h, \sigma}(\theta):= \Bigl( \sum_{k=0}^{\infty} \frac{\vert N_{h,\sigma,k}^{\pm}(\theta)\vert^2}{ (1+\lambda_k)^r}
\Bigr)^{1/2}\,\,.
\end{equation}
Let $N^{\pm}_{h, \sigma}(w)$ denote the non-tangential maximal function for the holomorphic function $\<{\Cal R}^{\pm}_{J,\sigma}(z)w,F^{\pm}_\sigma\>_h$.  By formulas \eqref{eq:Rbound} and \eqref{eq:Ndef}, it follows that, for all $\theta\in  \T$ and all functions $w \in {\bar H}^r_h(M)$, we have
\begin{equation}
\label{eq:Nwbound}
N^{\pm}_{h,\sigma}(w)(\theta) \leq  N^{\pm}_{h,\sigma}(\theta) \,  \Vert w\Vert_r  \,\,.
\end{equation}
The functions $N^{\pm}_{h,\sigma} \in L^p( \T,{\mathcal L})$ for any $0<p<1$. In fact, by formula
\eqref{eq:Nkbound} and (following a suggestion of Stephen Semmes) by the `triangular  inequality'
for the space $L^{p/2}$ with  $0<p<1$, we have
\begin{equation}
\label{eq:NLpbound}
\vert N^{\pm}_{h,\sigma}  \vert^p_p  \leq (A_{\alpha,p} \, C_{h} )^p \bigl(\sum_{k=0}^{\infty}
\frac{1}{ (1+\lambda_k)^{pr/2}}\bigr) \, \Vert f\Vert _{-1} ^p \,<\, +\infty \,\,.
\end{equation}
The series in formula \eqref{eq:NLpbound} is convergent by the Weyl asymptotics
(Theorem \ref{thm:Weyl}) since $pr/2>1$. Let then
$$
B_h(p,r) := (A_{\alpha,p} \, C_{h} ) \bigl(\sum_{k=0}^{\infty}
\frac{1}{ (1+\lambda_k)^{pr/2}}\bigr)^{1/p}\,.
$$
By taking the non-tangential limit as $z\to -e^{\mp 2 \imath\theta} $ in the identity 
\eqref{eq:keyidentity}, formula \eqref{eq:Nwbound} implies that, for all $\theta\in  \T$
such that $N^{\pm}_{h,\sigma} (\pi\mp 2\theta)<+\infty$, 
$$
\vert \<{\partial}^{\pm}_{h,\sigma} v,F^{\pm}_\sigma\>_h \vert \leq N^{\pm}_{h,\sigma} (\pi\mp 2\theta) \,
 \Vert (S_{\theta} + \imath \sigma_\theta) v \Vert_r  \,\,,
 $$
hence the required estimates \eqref{eq:aprioribound} and \eqref{eq:Lpbound} are proved
with the choice of the function $A_{h,\sigma} (f,\theta):= N^+_{h,\sigma} (\pi-2\theta)$ or  $A_{h,\sigma} (f,\theta):= N^- _{h,\sigma} (\pi+2\theta)$, for all $\theta\in  \T$. 
\end{proof}

\begin{proof}[Proof of Theorem~\ref{thm:distsol}] By the estimate 
\eqref{eq:aprioribound} of Lemma~\ref{lemma:aprioribound},  the linear map given by
\begin{equation}
\label{eq:functional}
(S_{\theta} + \imath \sigma_\theta) v\to -\<f,v\>\,\,, \quad \hbox{ for all } v\in H_h^{r+1}(M)\,\,,
\end{equation}
is well defined and extends by continuity to the closure of the range ${\bar R}^r_\sigma(\theta)$ of the linear operator $S_{\theta}+ \imath \sigma_\theta $ in ${\bar H}^r_h(M)$. Let $\Cal U_\sigma(f)(\theta)$ be the extension uniquely defined  by the condition that $\Cal U_\sigma(f)(\theta)$ vanishes on the orthogonal complement of ${\bar R}^r_\sigma(\theta)$ in ${\bar H}^r_h(M)$. By construction, for almost all $\theta\in  \T$ the linear functional $u:=\Cal U_\sigma(f)(\theta)\in {\bar H}^{-r}_h(M)$ yields a distributional solution of the cohomological equation  $(S_{\theta}+ \imath\sigma_\theta) u=f$ whose norm satisfies the bound 
$$
\Vert \Cal U_\sigma(f)(\theta) \Vert_{-r} \leq A_{h,\sigma}(f,\theta)\,.
$$
By \eqref{eq:Lpbound} the $L^p$ norm of the measurable function $\Cal U_\sigma(f): \T \to {\bar H}^{-r}_h(M)$ satisfies the required estimate
$$
 \vert \Cal U_\sigma(f) \vert_{p}:= \left(\int_{ \T} \Vert \Cal U_\sigma(f)(\theta) \Vert_{-r}^p \,d\theta \right)^{1/p}
 \,\, \leq \,\,  B_{h} \, \Vert f \Vert_{-1} \,\,.
 $$
 
\end{proof}

\begin{theorem} 
\label{thm:CEdistribution} 
Let $h$ be an Abelian differential with minimal vertical foliation.
For any $r>2$ and $p\in (0,1)$ such that $pr >2$, there exists a constant $C_{h,p,r}>0$  such that, for all zero-average functions $f\in {\bar H}^{r-1}_h(M)$, for all $\sigma \in \R$ and for Lebesgue almost all
$\theta \in \T$, the twisted cohomological equation $(S_{\theta} +  \imath\sigma_\theta) u=f$ has a distributional solution $u_\theta \in {\bar H}^{-r}_h(M)$  satisfying the following estimate:
\begin{equation}
\label{eq:solboundtwo}
\left(\int_\T \Vert u_\theta \Vert_{-r}^p  d\theta \right)^{1/p}  \leq C_{h,p,r}\, \Vert f \Vert_{r-1}\,\,.
\end{equation}
\end{theorem}
\begin{proof}   Let $\Cal E=\{e_k\}_{k\in {\N}}$ be the orthonormal Fourier basis of the Hilbert space $L^2_h(M)$ described in \S \ref{sec:analysis}. Let $r>2$ and $p\in (0,1)$ be such that $pr>2$.

\noindent  By Theorem \ref{thm:distsol}, for any $k\in \N\setminus\{0\}$ there exists a function with distributional values $u_k := \Cal U_\sigma(e_k)\in L^p\left( \T, {\bar H}^{-r}_h(M)\right)$ such that the following holds. There 
exists a constant $C_{h,r}:=C_h(p,r) >0$ such that
\begin{equation}
\label{eq:kpnorm}
\left( \int_{ \T}  \Vert u_k (\theta) \Vert_{-r}^p \, d\theta \right)^{1/p} \,\leq\,   C_{h,r}\, \Vert  e_k \Vert_{-1}    \, \leq \,C_{h,r} \,(1+\lambda_k)^{-1/2}  \,\,.
\end{equation}
In addition, for any $k\in \N\setminus\{0\}$, there exists a full measure set $\Cal F_k(\sigma)\subset  \T$ such that, for all $\theta\in {\Cal F}_k(\sigma)$, the distribution $u:=u_k(\theta) \in {\bar H}^{-r}_h(M)$ is a (distributional) solution of the cohomological equation $(S_{\theta} + \imath\sigma_\theta) u  = e_k$. 

\smallskip
\noindent Any function $f\in {\bar H}^{r-1}_h(M)$ of zero average has a Fourier decomposition in $L^2_h(M)$:
$$
f = \sum_{k\in\N\setminus\{0\}} \<f,e_k\>_h \, e_k \,\,.
$$
A (formal) solution of the cohomological equation $(S_{\theta} + \imath\sigma_\theta) u =f$ is therefore 
given by the series
\begin{equation}
\label{eq:solseries}
u_{\theta}:=  \sum_{k\in\N\setminus\{0\}} \<f,e_k\>_h\, u_k(\theta)\,\,.
\end{equation}
By the triangular inequality in ${\bar H}^{-r}_h(M)$ and by H\"older inequality, we have
$$
\Vert u_{\theta}\Vert_{-r} \leq  \left(\sum_{k\in\N\setminus\{0\}}   
\frac{ \Vert u_k(\theta)\Vert_{-r}^2}{(1+\lambda_k)^{r-1}}\right)^{1/2} \, \Vert f\Vert_{r-1}\,\,,
$$
hence by the `triangular inequality'  for $L^p$ spaces (with $0<p<1$) and by the estimate
\eqref{eq:kpnorm},
\begin{equation}
\label{eq:Lpsol}
\int_{ \T}  \Vert u_{\theta} \Vert_{-r}^p \, d\theta  \,\,\leq \,\, C_{h,r}^p \, \left( \sum_{k\in\N\setminus\{0\}}  \frac{1}{(1+\lambda_k)^{pr/2}} \right)  \, \Vert f\Vert_{r-1}^p\,\,.
\end{equation}
Since $pr/2>1$ the series in \eqref{eq:Lpsol} is convergent, hence  $u_{\theta}\in {\bar H}^{-r}_h(M)$ is a solution of the equation $(S_{\theta}+ \imath\sigma_\theta) u=f$ which satisfies the required bound \eqref{eq:solboundtwo}.

\end{proof}

\subsection{Twisted invariant distributions and basic currents} 
\label{TIDBC}

In this section we describe the structure of the space of obstructions to the existence of solutions of the twisted
cohomological equation $(S + \imath \sigma)u =f$.

\begin{definition}
\label{def:TID}
For all $r>0$, let  $\Cal I^r_{h,\sigma}  \subset H^{-r}_h(M)$ denote the space of distributions invariant for the twisted Lie derivative operator $S + \imath \sigma $, that is, the space
$$
\Cal I^r_{h,\sigma} :=\{ D \in  H^{-r}_h(M) \vert  (S + \imath \sigma ) D =0  \}\,.
$$
\end{definition}
\noindent  Twisted invariant distributions are in one-to-one correspondence with twisted basic currents.  We introduce Sobolev space of $1$-forms and of $1$-dimensional currents.  
\begin{definition}
For all $r>0$,  the weighted Sobolev space of $1$-forms $W^r_h(M)$ is defined  as follows: 
 \begin{equation}
W^r_h(M):= \{ \alpha \in L^2(M, T^\ast M)   \,\vert \, (\imath_S\alpha, \imath_T\alpha) \in 
H_h^r(M) \times H_h^r(M) \}\,.
\end{equation}
The weighted Sobolev space of $1$-currents $W^{-r}_h(M)$ is defined  as the dual spaces of the  weighted Sobolev space of $1$-forms $W^{r}_h(M)$.
\end{definition}
\noindent  Twisted basic currents are defined as follows:
\begin{definition}
\label{def:TBC} 
For all $r>0$, let  $\Cal B^r_{h,\sigma}\subset H^{-r}_h(M)$ denote the space of twisted basic currents, that is, the space
$$
\Cal B^r_{h,\sigma}:=\{ C \in  W^{-r}_h(M) \vert   (\mathcal L_{S} + \imath \sigma ) C = \imath_{S} C=0  \}\,.
$$
\end{definition} 
\noindent [Here $\mathcal L_{S}$  denotes  the Lie derivative operator on currents in the direction of the vector field $S$ on $M\setminus \Sigma_h$].

\noindent The notions of twisted invariant distributions and twisted basic currents are related (for the untwisted case see  \cite{F02}, Lemmas 6.5 and 6.6,  or \cite{F07}, Lemma 3.14):
\begin{lemma}
\label{lemma:DtoC}
A $1$-dimensional current $C\in {\Cal B}^r_{h, \sigma}$  if and only 
if the distribution $C\wedge \re (h)  \in {\Cal I}^r_{h,\sigma}$. In addition, 
the map
\begin{equation}
\label{eq:DtoC}
\Cal D_h: C  \to - C\wedge\re(h)
\end{equation}
is a bijection from the space ${\Cal B}^r_{h,\sigma}$ onto the space ${\Cal I}^r_{\sigma, h}$.
\end{lemma}
\begin{proof}
The map $\Cal D_h$ is well-defined since for any $C\in {\Cal B}^r_{h, \sigma}$ we have
$$
(\Cal L_S + \imath \sigma)[ C\wedge \re (h) ]= [(\Cal L_S + \imath \sigma) C] \wedge \re(h) =0\,.
$$
The inverse map is the map $\Cal B_h: {\Cal I}^r_{\sigma, h} \to {\Cal B}^r_{h,\sigma}$ defined as
$$
\Cal B_h (D) =  \imath_S D  \,, \quad \text{ for all }   D \in {\Cal I}^r_{\sigma, h}\,.
$$
[Here $\imath_S$ denotes the contraction operator with respect the vector field $S$ on $M\setminus \Sigma_h$,
which maps distributions, as  degree $2$ currents, to degree $1$ currents].

\noindent The map $\Cal B_h$ is well-defined since $\imath_S \circ \imath_S =0$, and 
$$
(\Cal L_S + \imath \sigma) \circ \imath_S  = \imath_S \circ  (\Cal L_S + \imath \sigma) \,.
$$ 
It follows that if  $D \in {\Cal I}^r_{h, \sigma}$ then $C=\imath_S D \in {\Cal B}^r_{h, \sigma}$ since
$$
(\Cal L_S + \imath \sigma) C =\imath_S \circ  (\Cal L_S + \imath \sigma)D =0 \quad \text{ and } \quad 
\imath_S C= \imath_S (\imath_S D) =0\,.
$$
Finally, the map $\Cal B_h$ is the inverse of the map $\Cal D_h$. In fact, since $\imath_S C=0$ (as $C$ is basic) and $D \wedge \re(h)=0$ (as a current of degree $3$), and $\imath_S \re (h)=1$, we have
$$
\begin{aligned}
(\Cal B_h \circ \Cal D_h) (C)& = -\imath_S(C \wedge \re (h) ) = -\imath_S C \wedge \re (h)  + C \wedge \imath_S \re h = C\,; \\ 
(\Cal D_h \circ \Cal B_h) (D) &= -( \imath_S D \wedge \re(h)) = -\imath_S (D \wedge \re(h)) + (D \wedge\imath_S \re(h))  = D\,.
\end{aligned}
$$
The argument is complete.
\end{proof}
\noindent We introduce the twisted exterior differential $d_{h,\sigma}$, defined on $1$-forms 
as
$$
d_{h,\sigma} \alpha :=d\alpha  + \imath \sigma \re (h) \wedge \alpha, \quad \text{ for all } \alpha \in W^{s}_h(M)\,.
$$
The twisted exterior differential extends to currents by duality. 
\begin{definition}
\label{def:TC} A current $C\in W^{-s}_h(M)$  is $d_{h,\sigma}$-closed if 
$$
d_{h,\sigma} C  = dC +\imath \sigma (\re (h) \wedge C )=0\,.
$$
\end{definition}

\begin{lemma} 
\label{lemma:basic} 
A current $C \in \Cal B_{h, \sigma}^r$ if and only if $\imath_S C=0$ and $C$ is  $d_{h,\sigma}$-closed.
\end{lemma} 
\begin{proof} If $C\in \Cal B_{h, \sigma} ^r$, then $\imath_SC=0$ by definition, and
$$
\imath_S d_{h,\sigma} C =  \imath_S [dC +\imath \sigma (\re(h) \wedge C )] =  \Cal L_S C +\imath \sigma C =0\,,
$$
so that $\imath_S d_{h,\sigma} C=0$, which implies $d_{h,\sigma} C=0$, as $d_{h,\sigma} C$ is a current of
degree $2$ (and dimension $0$) and the contraction operator $\imath_S$ is surjective onto $2$-forms.

\noindent Conversely, if $\imath_S C=0$ and $d_{h,\sigma} C=0$, then by the above formula 
$\Cal L_S C +\imath \sigma C =0$, hence $C \in \Cal B_{h, \sigma} ^r$, thereby completing the
argument.
\end{proof} 

\noindent By the de Rham theorem for twisted cohomology, it is possible to attach a twisted cohomology class to any  $d_{h,\sigma}$-closed current. 

\begin{definition}  
\label{def:TCohom}
Let $\Omega^*(D)$ denote the space of all smooth differential forms on a domain  $D\subset M$, with compact support on $D$. Let $\eta$ be a real closed smooth $1$-form on $D$ and let $d_\eta$ denote the twisted exterior derivative defined as 
$$
d_\eta \alpha  =  d\alpha  +  \imath \eta \wedge  \alpha \,, \quad \text{ for all } \alpha \in \Omega^*(M) \,.
$$
The twisted cohomology (with complex coefficients) $H^\ast _\eta(D, \C)$ is the cohomology of the differential 
complex $(\Omega^*(D), d_\eta)$. 
\end{definition} 

\noindent For every Abelian differential $h$ on $M$ and $\sigma \in \R$ , let us adopt the notation
$$
H^1_{h,\sigma} (M\setminus\Sigma_h, \C) :=  := H^1_{\sigma\re (h)} (M \setminus \Sigma_h, \C) \,.
$$

\begin{lemma}
\label{lemma:cohom_class} 
For every $r>0$, the exists a cohomology map $ j_r : \Cal B_{h, \sigma}^r \to H^1_{h, \sigma}(M, \C)$
such that $j_r(C)$ is the twisted cohomology class of the twisted basic current $C \in  \Cal B_{h, \sigma}^r$.
\end{lemma}
\begin{proof} A current $C\in W^{-r}_h(M)$ does not in general extend to a linear functional on $C^\infty(M)$,
hence is not a current on the compact surface $M$. However, since $C_0^\infty(M\setminus \Sigma_h) \subset W^r_h(M)$ for all $r>0$, it follows from the de Rham theorem for the twisted cohomology that  
current $C \in W^{-r}_h(M)$ such that $d_{h, \sigma} C=0$ in $W^{-(r+1)}_h(M)$ has a well-defined twisted cohomology class  $[C] \in H^1_{h, \sigma} (M\setminus \Sigma_h, \C)$.  
\end{proof} 

\noindent The structure of the space of basic currents with vanishing cohomology class, with respect to the filtration induced by weighted Sobolev spaces with integer exponent, was described 
in \cite{F02}, \S 7 (see also \cite{F07}, \S 3.3, with respect to the filtration induced by weighted Sobolev spaces with general real exponent).  We extend below these results  to the space of twisted basic currents.

\smallskip
\noindent Let $\delta_r: \Cal B^r_{h,\sigma} \to \Cal B^{r-1}_{h,\sigma}$ be the linear maps defined as follows 
(see \cite{F02}, formula $(7.18')$ and  \cite{F07}, formulas $(3.61)$ and $(3.62)$ for the untwisted case): 
\begin{equation}
\label{eq:deltas}
\delta_{r}(C) :=  (d_{h,\sigma} \circ \imath_T)( C )=  -\,d_{h,\sigma}  \left(\frac{C\wedge \re(h)}{\omega_h} \right)\,, \quad  \text{\text for }\,\,C\in  
\,\Cal B^r_{h, \sigma}\,.
\end{equation}
Indeed, it can be proved by Lemma \ref{lemma:DtoC} and by the definition of the weighted Sobolev spaces $H^r_h(M)$ and $W^r_h(M)$ that the above formula \eqref{eq:deltas} defines, for all $r>0$,
bounded linear maps $\delta_{r}: \Cal B^r_{h,\sigma} \to \Cal B^{r-1}_{h,\sigma}$.

\noindent  Let $\Cal K^{r}_{h, \sigma} \subset \Cal I^{r}_{h, \sigma} \subset  H^{-r}_h(M)$ denote the subspace of distributions which are twisted $S$-invariant and $T$-invariant, that is,
$$
\Cal K^{r}_{h, \sigma}:= \{ D \in H^{-r}_h(M) \vert  (S+\imath \sigma) D = TD =0 \} \,.
$$
Let $i_r: \Cal K^{r}_{h, \sigma} \to \Cal B^r_{h,\sigma}$ denote the restriction   to $\Cal K^{r}_{h, \sigma}$
of the inverse of the map $\Cal D^h:\Cal I^{r}_{h, \sigma}\to  \Cal B^r_{h,\sigma}$ (see Lemma~\ref{lemma:DtoC}) , that is, the map 
defined as 
$$
i_r (D) :=  \imath_S D  \,, \quad \text { for all }  D\in \Cal K^{r}_{h, \sigma} \,. 
$$

\begin{theorem} \label{thm:bcstruct}
For all $r>0$ there exist exact sequences
\begin{equation}
\label{eq:exactsequences}
0\to \Cal K^{r-1}_{h, \sigma} {^{\underrightarrow{\,\,\,\,i_{r-1}\,\,\,\,}} } {\Cal B}_{h, \sigma}^{r-1} {^{\underrightarrow{\,\,\,\,\delta_{r}\,\,\,\,}} }
{\Cal B}_{h,\sigma}^r {^{\underrightarrow{\,\,\,\,j_{r}\,\,\,\,}}}  H^1_{h, \sigma} (M\setminus \Sigma_h,{\C})\,\,.
\end{equation} 
  \end{theorem}
  \begin{proof} 
  The map $i_r: \Cal K^{r}_{h, \sigma} \to {\Cal B}_{h, \sigma}^{r}$ is by definition injective, since the contraction operator is surjective onto the space of functions ($0$-forms).  
  
 \smallskip 
\noindent  The identity $\text{Im}(i_r)= \text{ker} (\delta_r)$ holds since by Lemma~\ref{lemma:DtoC}  
 $C \in {\Cal B}_{h, \sigma}^r$ if and only if $C= \imath_S D$ with $D\in {\Cal I}_{h, \sigma}^r$
 and in addition
 $$
 \delta_r (\imath_S D)= d_{h,\sigma} (\imath_T \imath_S  D)=  \imath_T (S+ \imath \sigma)D  - \imath_S (T D) 
 = - \imath_S (T D)\,,
 $$
  hence $ \delta_r (\imath_S D)=0$ if and only if $TD=0$ (since $TD$ has degree $2$ and the contraction is
  surjective onto the space of functions ($0$-forms). 
  
  \smallskip 
 \noindent  The identity $\text{Im}(\delta_r)= \text{ker} (j_r)$ holds by the following argument. Let
  $C' \in \Cal B^s_{h,\sigma}$ be a current such that  $[C']=0 \in H^1_{h,\sigma}(M\setminus\Sigma_h, \C)$,  hence there exists a current $U$ of degree $0$  (and dimension $2$) such that $C'= d_{h,\sigma} U$.  Let $C= U \wedge \im(h)$.
  By definition we have $C'= \delta_r (C)$. We claim that $C= U \wedge \im(h) \in \Cal B^{s+1}_{h,\sigma}$.   
  In fact, by definition $\imath_S (U \wedge \im(h))=0$, and since $\imath_S C'=0$, we have
  $$
  \begin{aligned}
  (\Cal L_S + \imath \sigma) (U \wedge \im(h)) &=   (\Cal L_S + \imath \sigma) (U)  \wedge \im(h) 
  \\ &= \imath_S d_{h, \sigma} U   \wedge \im(h)  =  \imath_S C'   \wedge \im(h) =0\,.
\end{aligned}  
$$
The argument is thus complete. 
\end{proof}

\noindent The above theorem and Lemma \ref{lemma:DtoC} imply the following finiteness result:

\begin{corollary}
\label{cor:finiteness}
For any Abelian differential $h$ on $M$, for all $\sigma\in \R$ and  for all $s\geq 0$, the spaces $\Cal I^s_{h, \sigma}$ of twisted invariant distributions for the operator $S+ \imath \sigma \re h$  and the corresponding space 
of $\Cal B^s_{h, \sigma}$ of twisted basic currents have finite dimension.
\end{corollary}

\noindent We conclude the section by proving a lower bound on the dimensions of the spaces of twisted invariant
distributions.   

\begin{corollary}  Let $h$ be an Abelian differential with minimal vertical foliation. For all $\theta \in \T$, let $h_\theta := e^{-\imath \theta}h$ be the rotated Abelian differential and  let $\sigma_\theta:= \sigma_\theta$.  For any $r>2$ and for almost all $\theta \in \T$,  the subspace
$j_r (\Cal B^r_{h_\theta, \sigma_\theta}) \cap  H^1_{h_\theta,\sigma_\theta}(M, \C)$ has codimension at most equal to one in $H^1_{h_\theta,\sigma_\theta}(M, \C)$ .
\end{corollary} 
\begin{proof} Let $h_\theta := e^{-\imath \theta}h$ be the rotated Abelian differential, let $\sigma_\theta:= \sigma_\theta$  and let $\alpha$ be any twisted closed $1$-form, that is, a $1$-form such that
$$d_{h_\theta, \sigma_\theta} \alpha:=d\alpha  + \imath \sigma_\theta \re(h_\theta)\wedge  \alpha =0 \,.$$ 
Let $\alpha := f \re(h_\theta) +  g \im(h_\theta)$ and assume that $f\in \bar H^{r-1}_h(M)$ 
with $r>2$ and that $f$ is orthogonal to constant functions. Then by Theorem~\ref{thm:CEdistribution}   it follows that the cohomological
equation
$$
(S_\theta + \imath \sigma_\theta) u = f
$$
has a distributional solution $u\in \bar H^{-r}_h(M)$ for almost all $\theta \in \T$. Let $C$ denote the 
current of degree $1$ (and dimension $1$) uniquely determined by the formula
$$
d_{h_\theta, \sigma_\theta} u:= \left(d + \imath \sigma_\theta \re (h_\theta)\right) u= \alpha +   C \,. 
$$
It is clear from the definition that $C$ is closed with respect to the twisted differential
$d_{h_\theta, \sigma_\theta}$, that is,  $d_{h_\theta, \sigma_\theta} C=0$, and in addition $\imath_{S_\theta} C=0$, hence, by Lemma~\ref{lemma:basic}, the current $C$ is a twisted basic current $d_{h_\theta, \sigma_\theta}$-cohomologous to the $1$-form $\alpha$. 

\noindent  Finally, it can be proved that for all $\sigma\in \R$ all cohomology classes in $H^1_{h,\sigma}(M, \C)$ can be represented by  $d_{h, \sigma}$-closed $1$-forms $\alpha \in W^r_h(M)$ with $r>1$.

\end{proof}

\subsection{Smooth solutions} 
\label{SS}
In this section we prove our main result on existence of smooth solutions of the twisted cohomological equation for translation flows, which holds for {\it any} Abelian differential in almost all directions, and derive as a corollary
our result on cohomological equations for product translation flows.

\begin{lemma}
\label{lemma:aprioribound_smooth}
Let $h$ be an Abelian differential with minimal vertical foliation.   For every  $s>r \geq 0$ such that
$s-r >3$ there exists $p\in (0,1)$ such that  for every $\sigma \in \R$ there exists a function  
$A_{h, \sigma}:= A_{h,\sigma} (p,r,s)\in L^p(\T, {\mathcal L})$ such that the following holds. 
For almost all $\theta\in \T$ and for all zero average functions  $v\in  H^{s+1}_h(M)$,  we have
\begin{equation}
\label{eq:aproribound_smooth}
\vert   v  \vert_r  \leq A_{h,\sigma} (\theta) \,  \vert (S_\theta + \imath \sigma_\theta) v \vert_{s}\,,
\end{equation}
and there exists a constant $B_h:= B_h (p,r,s)>0$ such that, for all $\sigma \in \R$, we have
$$
\vert  A_{h,\sigma}   \vert_p   \leq  B_h\,.
$$
\end{lemma} 
\begin{proof}
Let $\Cal E := (e_k )_{k\in \N}$ denote the orthonormal Fourier basis of the space $L^2_h(M)$ of eigenvalues
of the Friedrichs extension of the flat Laplacian, described in \S  \ref{sec:analysis}.  Let $\alpha>2$ and let $p\in (0,1)$ be such that $\alpha p>2$. By Lemma~\ref{lemma:aprioribound}, for all $k\in \N\setminus \{0\}$ there exists 
a function $A^{(k)}_{h,\sigma}:= A^{(k)}_{h,\sigma}(p,\alpha) \in L^p(\T, {\mathcal L})$  such that, for all $v\in H^{\alpha+1}_h(M)$ of zero average we have
$$
\vert \<v,e_k\>\vert \leq A^{(k)}_{h,\sigma}(\theta)  \Vert   (S_\theta + \imath \sigma_\theta)v \Vert_\alpha \,.
$$
In addition, there exists a constant $B_h:= B_h(p,\alpha)$ such that
$$
\vert A^{(k)}_{h,\sigma} \vert_p \leq B_h \Vert  e_k \Vert_{-1} = B_h  (1+\lambda_k)^{-1/2} \,.
$$
Let $\beta>1$ such that $(\beta+1)p>2$. It follows that, for any $v \in H^{\alpha+1}_h(M)$ of zero average we have
$$
\Vert  v \Vert_{-\beta} \leq   \Bigl (\sum_{k\in \N\setminus \{0\} } (1+\lambda_k)^{-\beta} A^{(k)}_{h,\sigma}(\theta)^2  \Bigr )^{1/2}  \Vert   (S_\theta + \imath \sigma_\theta)v \Vert_\alpha \,.
$$
 Let then $A_{h,\sigma}:= A_{h, \sigma}(p,\alpha, \beta)$ denote the function defined, for $\theta\in \T$, as follows:
$$
 A_{h, \sigma}(\theta) := \Bigl (\sum_{k\in \N\setminus \{0\} } (1+\lambda_k)^{-\beta} A^{(k)}_{h,\sigma}(\theta)^2  \Bigr )^{1/2}  \,.
$$
By the triangular inequality for the space $L^{p/2}$  (with $p/2<1$) we have
$$
\begin{aligned}
\vert  A_{h,\sigma} \vert_p^p&\,= \,\left\vert \sum_{k\in \N\setminus \{0\} } (1+\lambda_k)^{-\beta} 
(A^{(k)}_{h,\sigma})^2 \right\vert_{p/2}^{p/2}  \\ &\leq \,  \sum_{k\in \N\setminus \{0\} }  (1+\lambda_k)^{-\beta p/2} \vert A^{(k)}_{h,\sigma} \vert_p^{p}   \,\leq \, B_h^p  \sum_{k\in \N\setminus \{0\} } 
 (1+\lambda_k)^{-(\beta+1)p/2} \,.
\end{aligned}
$$
By the Weyl asymptotics, the series on the RHS of the above formula is convergent as soon as
$(\beta+1)p > 2$.  Let then $v\in H^{\alpha+3}_h(M)$ so that $\Delta_h v \in H^\alpha _h(M)$ and we have 
$$
\begin{aligned}
\Vert  v \Vert_{-\beta+2} &= \Vert  (I-\Delta^F_h) v \Vert_{-\tau} \leq A_{h,\sigma}(\theta)  
\Vert   (S_\theta + \imath \sigma_\theta) (I-\Delta^F_h) v \Vert_\alpha   \\ &
=  A_{h,\sigma}(\theta)
\Vert (I-\Delta_h)  (S_\theta + \imath \sigma_\theta)  v \Vert_\alpha=A_{h,\sigma}(\theta) \Vert (S_\theta + \imath \sigma_\theta)  v \Vert_{\alpha+2} \,.
\end{aligned} 
$$
By the interpolation inequality for the Friedrichs norms and by Lemma \ref{lemma:comparison},  
for every $ \rho \in [0, 1)$,  whenever $\alpha p>2$, $(\beta+1)p >2$ and $\rho + \beta \leq 2$, 
we have
$$
\vert v \vert_{\rho} = \Vert  v \Vert_\rho\leq  A_{h, \sigma}(\theta) \Vert (S_\theta + \imath \sigma_\theta)  v 
\Vert_{\alpha+\beta+\rho} \leq A_{h, \sigma}(\theta) \vert (S_\theta + \imath \sigma_\theta)  v \vert_{\alpha+\beta+\rho} \,.
$$
Finally, for all $r\geq 0$, by applying the above bound to all functions $S^i T^j v$ and $T^iS^j$, 
for all $i,j \leq [r]$, we finally have that there exists a constant $C_r>0$ such that 
\begin{equation}
\label{eq:aprioribound_step}
\vert  v \vert_r \leq  C_r A_{h, \sigma} (\theta) \vert (S_\theta + \imath \sigma_\theta)  v \vert_{\alpha +\beta +r} \,.
\end{equation}
Since, given $s>r\geq 0$ with $s-r>3$ it is always possible to find $\alpha>2$, $\beta >1$ and
$p\in (0, 1)$ such that
$$
s= \alpha+ \beta, \quad \alpha p >2\,, \quad (\beta+1)p >1 \quad \text{ and } \quad \{r\} + \beta \leq 2\,,
$$
the bound in formula~\eqref{eq:aproribound_smooth} follows immediately from that in the above formula
\eqref{eq:aprioribound_step}, hence the argument is complete.

\end{proof}

\begin{theorem} 
\label{thm:GCEsmooth}
Let $h$ be an Abelian differential with minimal vertical foliation.  For any $s>r\geq 0$ such that $s-r>3$,
there exists $p\in (0, 1)$  and a constant $C_{r,s}>0$ such that the following holds.  For any  $\sigma \in \R$ and for almost all $\theta\in \T$,  for any $f\in H^{s}_h(M)$ of zero average such that  $\Cal D(f)=0$ for all 
twisted invariant distributions $\Cal D \in \Cal I^s_{h_\theta, \sigma_\theta}$, the cohomological equation 
$(S_{\theta} +\imath \sigma_\theta) u=f$ has a  zero-average solution $\Cal U_\theta(f) \in H^{r}_h(M)$ satisfying the following estimate:
\begin{equation}
\label{eq:GCEest1}
\left(\int_\T \vert \Cal U_\theta (f) \vert_{r}^p \, d\theta \right)^{1/p} \leq C_{r,s} \, \vert f \vert_{s}\,\,.
\end{equation}
\end{theorem} 
\begin{proof} It follows from the a priori bound of Lemma~\ref{lemma:aprioribound_smooth} that, for all 
$\sigma \in \R$ and for almost all $\theta\in \T$, the subspace 
$$
\{  f\in \Cal H^s_h(M) \vert  f \in   (S_\theta + \imath \sigma_\theta) [ {\Cal H}^r_h(M)] \} 
$$
is closed in $\mathcal H^s_h(M)$, hence it coincides with the kernel  of the subspace 
$\mathcal I^{-s}_{h_\theta, \sigma_\theta} \cap \Cal H^{-s}(M)$ of all twisted invariant distributions
vanishing on constant functions.  In addition, it follows by continuity that, there exists $p\in (0, 1)$ and
 a function $A_{h,\sigma} \in L^p(\T,{\Cal L})$ such that,  for all $f\in \Cal H^s_h(M) \cap
 \text{Ker} ( \mathcal I^{-s}_{h_\theta, \sigma_\theta} )$ the unique zero-average 
 solution $\Cal U_\theta(f) \in H^r_h(M)$ of the cohomological equation  $(S_\theta +  \imath \sigma_\theta) u =f$ satisfies the bound
 $$
 \vert  \Cal U_\theta(f) \vert_r \leq    A_{h, \sigma} (\theta) \vert  f \vert_s\,.
 $$
 From the above inequality and the bounds on the  $L^p$ norm of the function $A_{h, \sigma}$ established
 in Lemma~\ref{lemma:aprioribound_smooth}, it follows immediately that
 $$
\left( \int_\T  \vert  \Cal U_\theta(f) \vert_r^p d\theta \right)^{1/p}  \leq \vert A_{h, \sigma}\vert_p \, \vert  f \vert_s
\leq  B_h\, \vert  f \vert_s\,.
 $$
 The proof of the theorem is therefore complete. 
 
\end{proof}

\begin{proof} [Proof of Theorem~\ref{thm:main}]   The condition that $h$ has a minimal vertical foliation it is not restrictive since the statement is rotation invariant, and any Abelian differential has a minimal direction \cite{Ma}, \cite{AG}. 

\smallskip
\noindent If the function $f \in H^s_h(M)$ is constant, then for $\sigma \not =0$ the constant functions 
$u=-\imath f/\sigma$ is a solution (which is unique in $L^2_h(M)$ for almost all $\theta\in \T$. For $\sigma=0$, there is no solution unless $f=0$, in which case the solution is the zero constant. The argument is therefore reduced to the case of functions of zero average.

\smallskip
\noindent By Theorem~\ref{thm:GCEsmooth}, for any Abelian differential  $h$ with minimal vertical foliation,  the twisted cohomological equation $(S_\theta + \imath \sigma \cos\theta) u=f$ can be solved with Sobolev bounds for all $f \in H^s_h(M)$ of zero-average  in the kernel of all twisted invariant distributions, {\it  for all $\sigma \in \R$ and for almost all $\theta\in \T$}.    Let then $\mathcal F \subset \T\times \R$ be the set of $(\theta, \sigma) \in \T \times \R$ such that  the twisted cohomological equation $(S_\theta + \imath \sigma ) u=f$ can be solved with Sobolev bounds for all $f \in H^s_h(M)$ of zero average  in the kernel of all twisted invariant distributions. Since the map 
$(\theta, \sigma) \to (\theta, \sigma \cos \theta)$  from $\T \times \R$ into itself is absolutely continuous,  it follows from Theorem~\ref{thm:GCEsmooth} that the set $\mathcal F$ has full Lebesgue measure. Finally, the statement of the theorem follows by Fubini's theorem.

\end{proof}

\noindent For all $s, \nu \in \N$, let $H^{s, \nu}_h (M\times \T)$ denote the $L^2$ Sobolev space  on $M\times \T$ with respect to the invariant volume form $\omega_h \wedge d\phi$ and the vector fields $S$, $T$, and $\partial/\partial \phi$:  
 $$
 H^{s,\nu} _h(M\times\T):= \{ f \in L^2(M\times \T, d \text{vol}) \vert   \sum_{i+j\leq s}\sum_{\ell \leq \nu} \Vert S^i T^j \frac{\partial^\ell f}{\partial \phi^\ell} \Vert_0 < +\infty\}\,;
 $$
the space  $H^{-s, -\nu}_h(M\times\T)$ is defined as the dual space of $H^{s,\nu}_h(M\times\T)$.
 
\noindent Let now $V_{\theta, c}= S_\theta +  c\cos\theta \frac{\partial}{\partial \phi}$ denote a translation vector field on $M \times \T$, and let $\Cal I^{s,\nu}_{h_\theta, c}$ denote the space of $V_{\theta, c}$ 
invariant distributions.

\noindent The space $L^2(M\times \T, d \text{vol})$ of the product manifold with respect to
the invariant volume form $\omega_h \wedge d\phi$ decomposes as a direct sum of the eigenspaces
$\{H^0_n \vert n\in \Z\}$ of the circle action:
$$
L^2(M\times \T, d \text{vol}) = \bigoplus_{n\in \Z}  H^0_n \,.
$$
\begin{corollary}
\label{cor:3dim}
Let $h$ be an Abelian differential with minimal vertical foliation. Let $s>r \geq 0$ be such that $s-r>3$ and let 
$\nu>2$ and $\mu <\nu-2$. For all $c\in \R$ 
and for almost all $\theta \in \T$ there exists a constant  $C^{(\mu,\nu)}_{r,s} (\theta,c)>0$ such that 
the following holds. For any  $f \in  H^{s,\nu}(M\times \T)$ such that $D(f) =0$ for all $V_{\theta,c}$-invariant
distributions  $D \in \Cal I^{s,\nu}_{h_\theta,c} \subset H^{-s,-\nu}_h(M\times \T)$, the cohomological equation $V_{\theta,c} u=f$ has a  solution $u:=\Cal U_\theta(f) \in H^{r, \mu}_h(M)$ satisfying the following estimate:
\begin{equation}
\vert \Cal U_\theta (f) \vert_{r,\mu}\leq C^{(\mu,\nu)}_{r,s}(\theta,c)\, \vert f \vert_{s,\nu}\,\,.
\end{equation}
\end{corollary}
\begin{proof} By the Fourier decomposition with respect to the circle action, the argument is reduced
to proving the existence of solutions for the cohomological equations
\begin{equation}
\label{eq:CE_n}
(S_\theta +  2\pi \imath c n \cos\theta ) u_n =f_n\,.
\end{equation}
For $n=0$ the above equation reduces to the cohomological equation for the translation flow on $M$, so that
the result already follows from \cite{F07}.
For every $n\in \N\setminus \{0\}$,  let $\sigma^{c,n}:= 2\pi c n\in \R$. The (finite dimensional) space $\Cal I^s_{h_\theta, \sigma^{c,n}_{\theta}} \subset H^{-s}_h(M)$ of twisted $(S_\theta +  \imath \sigma^{c,n}_\theta)$-invariant distributions embeds as a subspace of the space $\Cal I^{s, \nu}_{h_\theta,c} \subset H^{-s, -\nu}_h(M\times \T)$ of $V_{\theta,c}$-invariant distributions, for all $\nu\in \N$, by the formula
$$
D \left ( \sum_{n\in \Z}  f_n  e^{2\pi \imath \phi}       \right)  =  D(f_n)\,.
$$
By Theorem \ref{thm:GCEsmooth} there exist  constants $C_{r,s}>0$ and $p\in(0,1)$ and, for every $\epsilon >0$, there exists a full measure 
set $\Cal F_{c,n}(\epsilon) \subset \T$ of measure  at least $1 -\epsilon/n^2$, such that  for all 
$\theta \in \Cal F_{c,n}(\epsilon)$, for every $f_n \in \Cal I^s _{h_\theta, \sigma^{c,n}_\theta}$ there exists a solution $u_n\in H^r_h(M) \cap H^0_n$ of the cohomological equation~\eqref{eq:CE_n} which satisfies the Sobolev estimate
$$
\vert  u_n \vert _r \leq   C_{r,s}  \epsilon ^{-1/p} n^{2/p}  \vert  f_n \vert_s \,.
$$
In fact, the above claim follows immediately from Theorem~\ref{thm:GCEsmooth}.

\noindent From the claim it follows that  for all functions $f \in H^{s,\nu}_h(M\times \T)$ with $\nu >2/p$,  such that 
$f_n \in \Cal I^{s,\nu} _{h_\theta, \sigma^{c,n}_\theta}$ for all $n\not =0$, for all $\theta \in \Cal F_{c}(\epsilon):=  
\cap_{n\not=0}  \Cal F_{c,n}(\epsilon)$ the function $u = \sum_{n\not =0} u_n \in H^{r, \nu - 2/p }(M)$ 
is a solution of the cohomological equation  $V_{\theta, c} u =f$. Since, for any $\epsilon>0$, the set 
$\Cal F_{c}(\epsilon)$ has Lebesgue measure at least $1- C \epsilon$ with $C=\sum_{n\not=0} 1/n^2$,
the argument is complete.
\end{proof}

\begin{proof} [Proof of Theorem~\ref{thm:main_bis}] 

The space $\Cal I^{s,\nu}_{h_\theta,c}$ of $V_{\theta, c}$-invariant  distributions is generated by the
union of subspaces $\Cal I^s _{h_\theta, \sigma^{c,n}_\theta}$ over all $n\in \Z$. The statement of the theorem then follows from Corollary~\ref{cor:3dim} by Fubini's theorem.
\end{proof}

\begin{proof} [Proof of Corollary~\ref{cor:main}]
For any $\phi_0\in \T$, let $M_{\phi_0} = M \times \{\phi_0\} \subset M\times \T$. The return map of the flow
of the vector field $X_{\theta,c}$ to the transverse surface $M_{\phi_0}\equiv M$ is smoothly conjugate to the
time-$1/c$ map $\Phi_{\theta,c}^{1/c}$ of the translation flow generated by the horizontal vector field $S_\theta$ 
on $M$.  Since the return time function is constant (equal to $1$), it is possible to derive results on the cohomological equation for the return (Poincar\'e) map (the time-$1/c$ map) from results on the cohomological
equation for the flow. In fact, the procedure is as follows.  Let $\Phi_{\theta, c}^\R$ denote the flow of
the vector field $X_{\theta, c}$ on $M\times \T$. Let $ \chi \in C^\infty(\T)$ be a smooth function with
integral equal to $1$ supported on a closed interval $I \subset \T\setminus \{\phi_0\}$. Let $F(f)\in H^{s, \infty}_h(M\times \T)$ be the function defined as follows:
$$
F(f) \circ \Phi_{\theta, c}^t (x, \phi_0)  =   \begin{cases}    f(x) \chi (t)\,,     \quad &\text{ for }   t \in I, \\
0 \,,     \quad &\text{ for }   t \not \in I \,.  \end{cases} 
$$
Let  $f\in H^s_h(M)$ and let us assume that  the cohomological equation 
$ X_{\theta, c} U = F(f)$ has a solution  $U\in H^{r, \mu}_h(M\times \T)$. Then the restriction
$u = U\vert M_{\phi_0}$ is a solution of the cohomological equation $u \circ \Phi^{1/c}_{\theta,c} -u =f$ for the
time-$1/c$ map. In fact, for all $x\in M$, we have
$$
u \circ \Phi^{1/c}_{\theta,c} -u  = \int_0^{1/c}  X_{\theta, c} U \circ \Phi_{\theta, c}^t  dt  = 
\int_0^{1/c}  f  \chi (\phi_0 + c t)  dt = f(x) 
$$
By the Sobolev trace theorem, for any $\mu > 1/2$, the restriction $U\vert M_{\phi_0}$ of a function 
$U\in H^{r, \mu}_h (M\times \T)$ is a function $u\in H^r_h(M)$ and there exists $C_\mu>0$
such that
$$
\vert u  \vert_r  \leq  C_\mu  \Vert  U  \Vert_{ H^{r,\mu}_h(M\times \T)}\,.
$$
The result then follows from Theorem~\ref{thm:main_bis}. In fact, for every $X_{\theta,c}$-invariant distribution 
$D \in \Cal I^{s, \nu}_{h_\theta,c} \subset H^{-s,-\nu}_h(M\times \T)$ we define the distribution $D_M \in H^{-s}_h(M)$ as
$$
D_M(f)  :=  D (F(f) ) \,.
$$ 
By Theorem~\ref{thm:main_bis} it follows that, for $D_M(f)=0$ for all $D \in \Cal I^{s, \nu}_{h_\theta,c}$
then there exists a solution $U\in H^{r, \mu}_h(M\times \T)$ of the cohomological equation $X_{\theta,c} U=F(f)$,
hence a solution $u=U\vert M \in H^r_h(M)$ of the equation  $u\circ \Phi^{1/c} _\theta -u = f$. 
Finally, we have that for all $u\in H^{\infty}_h(M)$ such that $u\circ \Phi^{1/c} _\theta -u \in H^s_h(M)$ we have
$$
D_M (u\circ \Phi^{1/c} _{\theta,c} -u )= D (F(u\circ \Phi^{1/c} _{\theta,c} -u)) =D( F(u) \circ \Phi_{\theta, c}^{1/c} -  F(u) ) =0 \,,
$$
since $D \in \Cal I^{s, \nu}_{h_\theta,c}$ is by assumption $X_{\theta, c}$-invariant. 

\end{proof}


\begin{thebibliography}{20}
\bibitem[Ad]{Ad} A.~Adam, Transfer operators and horocycle averages on closed manifolds, preprint, 
arXiv:1809.04062.
\bibitem[AG]{AG} S.~H.~Aranson and V.~Z.~Grines, On some invariants of dynamical 
systems on two-dimensional manifolds (necessary and sufficient conditions
for the topological  equivalence of transitive dynamical systems),
{\it Mat. Sb.} \textbf {90}(132) (1973), 372-402 (Russian). Transl. in {\it Math. USSR Sb.}
\textbf {19} (1973), 365--393. 
\bibitem[BS14]{BS14} A.~I.~Bufetov and B.~Solomyak, On the modulus of continuity for spectral measures in substitution dynamics,  {\it Adv. Math.} {\bf 260} (2014), 84--129.
\bibitem[BS18a]{BS18a} \bysame, The H\"older property for the spectrum of translation flows in genus two, {\it Israel J. Math.}  {\bf 223} (1) (2018), 205--259.
\bibitem[BS18b]{BS18b} \bysame, On ergodic averages for parabolic product flows, {\it Bull. Soc. Math. France}
\textbf{146} (4) (2018), 675--690.
\bibitem[BS18c]{BS18c} \bysame, A spectral cocycle for substitution systems and translation flows, preprint, arXiv:1802.04783.
\bibitem[BS19]{BS19} \bysame, H\"older regularity for the spectrum of translation flows, preprint,  arXiv:1908.09347.
\bibitem[CE]{CE}  J.~Chaika and A.~Eskin, Every flat surface is Birkhoff and Oseledets generic in almost every direction, {\it J. Mod. Dynam.}  \textbf{9} (2015),  1 -- 23.
\bibitem[EM]{EM} A.~Eskin and M. Mirzakhani , Invariant and stationary measures for the $\SL(2,\R)$ action on Moduli space, {\it Publications math\'ematiques de l'IH\'ES} \textbf{127} (1) (2018), 95--324.
\bibitem[EMM]{EMM} A.~Eskin, M. Mirzakhani  and A. Mohammadi, Isolation, equidistribution, and orbit closures for the $\SL(2,\R)$ action on moduli space,   {\it Annals of Mathematics}   \textbf{182} (2015), 673--721.
\bibitem[Ft]{Ft} P.~Fatou, S\'eries trigonom\'etriques et s\'eries de Taylor,
{\it Acta Math.} \textbf {30} (1906), 335--400.
\bibitem[FGL]{FGL} F. Faure, S. Gou\"ezel and E. Lanneau,  Ruelle spectrum of linear pseudo-Anosov maps, {\it Journal de l'\'Ecole polytechnique--Math\'ematiques} \textbf{6} (2019), 811--877.
 \bibitem[FG18]{FG18} F.~Faure and C.~Guillarmou, Horocyclic invariance of Ruelle resonant states for contact Anosov flows in dimension $3$, {\it Math. Res. Lett.}
\textbf{25} (5) (2018), 1405 --1427.
\bibitem[Fi]{Fi} S.~Filip, Semisimplicity and rigidity of the Kontsevich-Zorich cocycle, {\it Invent. Math.} \textbf{205} (3), (2016), 617--670.
\bibitem[FlaFo03]{FlaFo03}  L.~Flaminio and G.~Forni, Invariant Distributions and Time
    Averages for Horocycle Flows, {\it Duke Math.  J.}, \textbf {119} (3) (2003), 465--526.
\bibitem[FlaFo06]{FlaFo06} \bysame,   Equidistribution of nilflows and applications to theta sums, 
{\it Ergodic Theory Dynam. Systems}~\textbf{26} (2) (2006), 409--433.
\bibitem[FlaFo07]{FlaFo07} \bysame,  On the cohomological equation for nilflows, {\it J. Mod. Dynam.}
\textbf{1} (1)  (2007), 37--60.
\bibitem[FlaFo14]{FlaFo14} \bysame, On effective equidistribution for higher step nilflows, preprint,  arXiv:1407.3640v1.
\bibitem[FFT16]{FFT16}  L.~Flaminio, G.~Forni and J. Tanis, Effective equidistribution of twisted horocycle flows and horocycle maps, {\it Geometric and Functional Analysis} \textbf{26 (5)},1359--1448.
\bibitem[F97]{F97} G.~Forni, Solutions of the Cohomological Equation for Area-Preserving Flows on Compact Surfaces of Higher Genus,
 {\it Ann. of Math.} \textbf{146}(2) (1997), 295--344. 
 \bibitem[F02]{F02}
\bysame, Deviation of ergodic averages for area-preserving flows on
  surfaces of higher genus, {\it Ann. of Math.} (2) \textbf{155} (2002), no.~1,
  1--103.
 \bibitem[F07]{F07} \bysame, Sobolev regularity of solutions of the cohomological equation,  {\it Ergodic Theory and Dynamical Systems}, 1-105. doi:10.1017/etds.2019.108.
 \bibitem[F19]{F19} \bysame, Twisted translation flows and effective weak mixing, preprint, arXiv:1908.11040v1.
 \bibitem[GL]{GL} P.~Giulietti and C.~Liverani,  Parabolic dynamics and Anisotropic Banach spaces, 
 {\it J. Europ. Math. Soc.} \textbf{21} (9) (2019), 2793--2858. Published online: 2019-05-20
 DOI: 10.4171/JEMS/892.
 \bibitem[GT12]{GT12} B.~Green and T.~Tao, The quantitative behaviour of polynomial orbits on nilmanifolds, {\it Ann. of Math.} \textbf{175} (2012), 465--540.
 \bibitem[HL]{HL} G.~H.~Hardy and J.~E.~Littlewood, A maximal theorem with
function-theoretic applications, {\it Acta Math.} \textbf {54} (1930), 81--116.
\bibitem[LM]{LiMa}
J.~L. Lions and E.~Magenes, {\it Probl\`emes aux limites non homog\`enes et
  applications, vol. 1}, Dunod, Paris, 1968.
  \bibitem[Ma]{Ma} A.~G.~Maier, Trajectories on the closed orientable surfaces, {\it Mat. Sb.} 
\textbf {12(54)}(1943), 71--84 (Russian).
\bibitem[MMY05]{MMY05} S. Marmi, P. Moussa and J.-C. Yoccoz, The cohomological equation for Roth-type
interval exchange maps, {\it J. of the AMS}  \textbf{18} (4) (2005), 823--872.
\bibitem[MY16]{MY16} S. Marmi and J.-C. Yoccoz,    H\"older Regularity of the Solutions of the Cohomological Equation for Roth Type Interval Exchange Maps, {\it Commun. Math. Phys.}  \textbf{344} (1)  (2016) 117--139. https://doi.org/10.1007/s00220-016-2624-9
  \bibitem[Nel59]{Ne59} E.~Nelson, Analytic vectors, {\it Ann. of Math.} (2) \textbf{70} (1959),
  572--615.
\bibitem[Rz]{Rz} F.~Riesz, \"Uber die Randwerte einer analytischen Funktion,
{\it Math. Z.} \textbf {18} (1923), 87--95.
\bibitem[Rd]{Rd}  W.~Rudin, {\it Real \& Complex Analysis},  McGraw-Hill Book Company, New York, 1987  (third edition).
\bibitem[Sm]{Sm} V.~Smirnov, Sur les valeurs limites des fonctions r\'eguli\`eres
\`a l'int\'erieur d'un cercle, {\it J. de la Soc. Physico-Math\'ematique de
Leningrad} \textbf{2} (1929), 22--37.
\bibitem[SW]{SW}  E.~M.~Stein and G.~Weiss, {\it Introduction to Fourier
Analysis on Euclidean Spaces}, Princeton Univ. Press, Princeton, 1971.
\bibitem [Ta12]{Ta12} J.~Tanis, The Cohomological Equation and Invariant
 Distributions for Horocycle Maps,  {\it Ergodic Theory and Dynamical
 systems}, \textbf{12} (2012), 1--42. 
\bibitem[Yo]{Yo} K.~Yosida, {\it Functional Analysis},  Springer-Verlag, New York Berlin, 1980  (sixth edition).
\bibitem[Zy]{Zy} A.~Zygmund, {\it Trigonometric Series}, Cambridge Univ. Press, Cambridge, 1959 
(reprinted in 1990).
\end{thebibliography}
\end{document}